\documentclass[10pt,a4paper]{article}

\usepackage{amsmath}
\usepackage{tocbibind}

\usepackage{mathtools}
\usepackage{faktor}
\usepackage{amsthm} 
\usepackage{amssymb} 

\usepackage[margin=2.5 cm]{geometry}
\usepackage{enumerate}

\usepackage[colorlinks={true},linkcolor={blue},citecolor={blue}, urlcolor={blue}]{hyperref}
\usepackage[noabbrev,nameinlink]{cleveref}
\usepackage{tikz}
\usetikzlibrary{calc,decorations.markings,arrows.meta,shapes.misc}
\theoremstyle{plain}

\newtheorem{thm}{Theorem}[section]
\newtheorem*{thm*}{Theorem}
\newtheorem{prop}[thm]{Proposition}
\newtheorem{cor}[thm]{Corollary}
\newtheorem{lem}[thm]{Lemma}

\newtheorem{mainthm}{Theorem}

\newtheorem{maincor}{Corollary}

\crefname{mainthm}{theorem}{theorems}
\Crefname{mainthm}{Theorem}{Theorems}
\crefname{maincor}{corollary}{corollaries}
\Crefname{maincor}{Corollary}{Corollaries}

\theoremstyle{definition}
\newtheorem{defn}[thm]{Definition}
\newtheorem{conj}[thm]{Conjecture}

\newtheorem{rem}[thm]{Remark}

\DeclareMathOperator{\Tr}{Tr}
\DeclareMathOperator{\rk}{rk}
\DeclareMathOperator{\Aut}{Aut}
\DeclareMathOperator{\diam}{diam}
\DeclareMathOperator{\Sing}{Sing}

\newcommand{\CC}{\mathbb{C}}

\author{Paolo Dolce}
\title{Kobayashi length bounds on bordered surfaces and generalized integral points on abelian varieties}
\date{}
\newcommand{\Addresses}{{%
  \bigskip
  \bigskip
  \footnotesize
  P.~Dolce, \textsc{Institute for Theoretical Sciences, Westlake University, China}\par\nopagebreak
  \textit{E-mail address}: \texttt{dolce@westlake.edu.cn}%
}}

\begin{document}
\maketitle

\begin{abstract}
Let $B$ be a compact Riemann surface and $B_0\subseteq B$ a hyperbolic surface obtained by removing finitely many disjoint closed disks. Fix a nontrivial loop $\alpha$ in $B_0$. For $s\ge0$, let $L(\alpha,s)$ denote the supremum, over all finite subsets $S\subset B_0$ with $\#S\le s$, of the minimal Kobayashi length of a loop in $B_0\smallsetminus S$ that is freely homotopic to $\alpha$ in $B_0$. Phung in \cite{Phung} proved that $L(\alpha,s)$ grows at most linearly and at least as $\sqrt{s}/\log s$. We sharpen the upper bound to $O(\sqrt{s+1})$, which determines $\lim_{s\to\infty}\log L(\alpha,s)/\log s=1/2$, answering \cite[Question~1.4]{Phung}. As an application, we improve Phung's counting bound for generalized integral points on abelian varieties over complex function fields. For an abelian variety of dimension $n$ over $\mathbb C(B)$ satisfying his geometric hypotheses, Phung proved a bound of order $(s+1)^{2nk}$ modulo the constant trace, where $k$ is the minimal number of generators of $\pi_1(B_0)$. We explain his method in detail and insert our stronger length estimate to obtain $(s+1)^{nk}$, halving the exponent. Counting in the Mordell--Weil lattice of rank $r$ gives the more precise bound $O((s+1)^{r/2})$ and a linear bound in $s+1$ for the canonical heights, with the family and $B_0$ fixed. For normal-crossings $D$ with Cartier model closure, this proves the fixed-model geometric Lang--Vojta conjecture in the stated abelian setting, without an exceptional set.
\end{abstract}
\tableofcontents

\section{Introduction}
The main result of this paper is an upper bound for the Kobayashi length of representatives of a fixed homotopy class on a bordered Riemann surface with finitely many punctures. The bound is uniform over all configurations of the punctures and, together with Phung's lower bound, answers his question about the polynomial growth exponent. We first present this analytic result and its proof strategy, then explain its application to generalized integral points through Phung's implementation of the Parshin cocycle method.

\subsection{Kobayashi lengths on punctured bordered surfaces}\label{sec:intro-lengths}
Let $B$ be a compact connected Riemann surface and let $\overline\Delta_1,\ldots,\overline\Delta_t$ be finitely many pairwise disjoint closed topological disks in $B$, allowing $t=0$. Put
\[
B_0=B\setminus\bigsqcup_{i=1}^t\overline\Delta_i
\]
and assume that $B_0$ is hyperbolic. The disks and the surface are fixed, while the finite puncture set will vary. We specify the distinction between the removed disks and the metric balls used in the estimates.

Fix a smooth Riemannian metric $\rho$ on $B$, with length $\ell_\rho$ and distance $d_\rho$. We use two different notions of disk and keep their notation and roles separate.
\begin{itemize}
\item A \emph{closed (topological) disk} is an embedded subset $\overline\Delta\subset B$ homeomorphic to the closed unit disk, with piecewise-$C^1$ boundary. These are the disks $\overline\Delta_1,\ldots,\overline\Delta_t$ removed from $B$ to form $B_0$. They are not required to be metric balls; this flexibility is needed in \S\ref{sec:phung_proof}, where disks are enlarged by collars and by joining auxiliary disks with thin tubes.
\item A \emph{metric ball} is defined by the fixed Riemannian distance:
\[
\Delta_\rho(x,r)=\{b\in B:d_\rho(b,x)<r\},\qquad
\overline\Delta_\rho(x,r)=\{b\in B:d_\rho(b,x)\le r\}.
\]
These balls are used in the local metric and area estimates of \S\ref{sec:growth}. For $r$ sufficiently small (in particular, below the injectivity radius at $x$), the closed metric ball is a closed topological disk with smooth boundary. 
\end{itemize}
Thus $\overline\Delta_i$ denotes a removed topological disk, while the subscript $\rho$ and the arguments $(x,r)$ identify a metric ball. 
When $t\ge1$, $B_0$ is the interior of a compact bordered surface and is itself an open Riemann surface without boundary. When $t=0$, it is the compact surface $B$.

We study the Kobayashi length of loops on punctured bordered surfaces. For a finite $S\subset B_0$, let $\ell_S$ be the length induced by the Kobayashi metric of $B_0\setminus S$. The inclusion into $B_0$ decreases Kobayashi lengths, so a path avoiding $S$ can only become longer when measured in the punctured surface. Geometrically, the punctures create new cusps. It is natural to ask how short a representative of a fixed homotopy class can remain in the worst configuration of at most $s$ punctures. For a nontrivial loop $\alpha$ in $B_0$, set
\[
L(\alpha,s)
=
\sup_{\substack{S\subset B_0\\ \#S\le s}}
\inf\bigl\{
\ell_S(\gamma)\colon
\gamma\text{ is a loop in }B_0\setminus S
\text{ freely homotopic to }\alpha\text{ in }B_0
\bigr\}.
\]
\begin{mainthm}\label{thm:new_bound_L}
For every such $\alpha$ there is a constant $C>0$, depending only on $\alpha$, $B_0$ and the fixed background metric, such that
\begin{equation}\label{eq:new_bound_L}
L(\alpha,s)\le C\sqrt{s+1}\qquad(s\in\mathbb Z_{\ge0}).
\end{equation}
\end{mainthm}

Phung proved a linear upper bound and a lower bound $c\sqrt{s}/\log(s+2)$ \cite[Theorems~B and~C]{Phung}. Together with Theorem~\ref{thm:new_bound_L}, the latter implies
\begin{equation}\label{eq:exact_growth}
\lim_{s\to\infty}\frac{\log L(\alpha,s)}{\log s}=\frac12,
\end{equation}
answering \cite[Question~1.4]{Phung}. 

The main idea of the proof of \Cref{thm:new_bound_L} is to average
lengths over a family of nearby loops. Suppose first that $\alpha$
has a smooth simple representative. Moving this loop slightly
in the normal direction gives a family of parallel loops filling
a thin tubular neighborhood. Phung's local estimates control how
much puncturing at $S$ increases the Kobayashi metric. Although
this increase can be arbitrarily large near the punctures, the
region where it exceeds a large value $v$ has area at most
$Cs/v^2$, where $s=\#S$. Its area is also bounded by that of the
fixed neighborhood. Integrating the minimum of these two bounds
gives an average length of order $\sqrt{s+1}$. Since only finitely
many parallel loops meet $S$, one avoids the punctures and has
the required length, without changing its free homotopy class
in $B_0$.

For a general class, we write $\alpha$ as a word in simple
generators of $\pi_1(B_0,b_0)$. We arrange their representatives
to have the same tangent direction at $b_0$, so that moving them
normally gives loops with a common moving base point.
Averaging the sum of their lengths makes all generators short
simultaneously, and concatenation gives the bound for $\alpha$.
The common motion of the base point also supplies the single
conjugation needed in Phung's counting argument.
The compact case is handled by removing an auxiliary disk.

We now recall the geometric Lang--Vojta setting and Phung's
generalized integral points, before explaining how these length
bounds enter his counting method and yield a linear height bound.

\subsection{The Lang--Vojta conjecture and generalized integral points}\label{sec:intro-integral}
Let $k$ be an algebraically closed field of characteristic zero and let $B$ be a smooth projective connected curve over $k$, with function field $K=k(B)$. Let $X$ be a smooth integral projective $K$-variety and let $D$ be a reduced effective divisor on $X$. A model $(\mathcal X,\mathcal D)$ consists of a normal $k$-variety with a proper flat morphism $f\colon\mathcal X\to B$, an identification $\mathcal X_K\simeq X$, and the reduced horizontal closure $\mathcal D$ of $D$. Such models exist.

We distinguish the divisor on the generic fibre from its closure on the total space. Since $X$ is smooth, every Weil divisor on $X$, in particular $D$, is Cartier. The total space $\mathcal X$, however, is only assumed normal and may be singular, so its Weil divisor $\mathcal D$ need not be Cartier.  Whenever we use $\deg_B\sigma_P^*\mathcal D$, we assume that $\mathcal D$ is Cartier. This is automatic on a regular model, which can be chosen projective in characteristic zero. This assumption concerns the total-space divisor and is unnecessary for the incidence conditions and counting statements below.

Properness and the valuative criterion identify $X(K)$ with the sections of $f$; write $\sigma_P$ for the section associated to $P$. Put
\[
\Sing(f)=\{b\in B:\mathcal X_b\text{ is singular}\}.
\]
For a finite set $S\subset B$, the set of $(S,\mathcal D)$-integral points is
\[
\mathcal X(\mathcal O_{S,\mathcal D})
=\{P\in X(K):f(\sigma_P(B)\cap\mathcal D)\subseteq S\}.
\]
In other words, these are precisely the sections whose image does not intersect $\mathcal D$ ``above'' $B\smallsetminus S$. In particular, an integral section is not contained in $\mathcal D$. If $\mathcal D$ is Cartier, its pullback along such a section is an effective divisor on $B$, and we have the usual notion of height of a point:
\begin{equation}\label{eq:divisor-model-height}
h_{\mathcal D}(P):=\deg_B\sigma_P^*\mathcal O_{\mathcal X}(\mathcal D)
=\deg_B\sigma_P^*\mathcal D.
\end{equation}

For a smooth projective pair  $(X, D)$ with normal-crossings boundary divisor $D$, \emph{log general type} means that $\omega_X(D)$ is big. 

\begin{conj}[Geometric Lang--Vojta, fixed-model boundary-degree form]\label{conj:lang-vojta}
Suppose that the pair $(X,D)$ is of log general type. Fix a model $(\mathcal X,\mathcal D)$ with $\mathcal D$ Cartier. There exist a proper closed subset $Z\subset X\setminus D$ and a constant $m>0$, depending on the fixed pair, model and base curve $B$, such that for every finite $S\subset B$ and every
$P\in\mathcal X(\mathcal O_{S,\mathcal D})$ with $P\notin Z$,
\begin{equation}\label{eq:lang-vojta}
\deg_B\sigma_P^*\mathcal D\le m\max\{1,2g(B)-2+\#S\}.
\end{equation}
\end{conj}

Roughly speaking, every integral point either lies in a fixed proper exceptional subset or has intersection degree bounded linearly in $\#S$, with the contribution of the base expressed by $2g(B)-2$. The constants in this formulation are independent of $S$ and $P$, while the model and $B$ remain fixed. There are stronger uniform versions of the geometric Lang--Vojta conjecture, requiring independence of the base curve under the specified geometric data and base changes. We use only the fixed-model formulation above.

Earlier results include Hindry--Silverman's height bounds for elliptic curves \cite{HS} and Buium's $abc$ theorem for abelian varieties with zero trace \cite{Buium}. Here an $abc$-type bound controls heights linearly by the number of distinct boundary intersections, up to an additive constant. For constant abelian varieties and constant divisors, Buium \cite{Buium98} and Noguchi--Winkelmann \cite{NWCurves} established bounds for intersection multiplicities under their respective hypotheses; summing these bounds gives an intersection-degree bound linear in the number of distinct intersection points.

Corvaja--Zannier proved cases of the geometric Lang--Vojta conjecture in the split setting, for certain plane complements and ramified covers of $\mathbb G_m^2$ \cite{CZ08,CZ13}. Their related $abcd$ theorem gives lower bounds for the number of distinct zeros of $1+u+v$ in terms of the degrees of the rational functions $u,v$, when these have few distinct zeros and poles \cite{CZ11}. Nonsplit extensions were obtained by Turchet \cite{Turchet17} and Capuano--Turchet \cite{CT}; see also Guo--Nguyen--Sun--Wang \cite{GNSW} for higher-dimensional toric results. The precise assumptions and the available uniformity differ among these results; for background, see \cite[\S1]{Phung} and \cite[\S1.1]{CT}. The results below concern a fixed complex function field and a fixed family. They do not require the boundary to have normal crossings; their geometric hypothesis on the boundary will be stated separately.

For the application we now take $k=\mathbb C$ and use the surface notation introduced in \S\ref{sec:intro-lengths}. The notion of $(S,\mathcal D)$-integral point can be generalized by allowing $S$
to vary while keeping only its cardinality $s=\#S$ bounded, and by requiring the intersection condition to hold only on the complement of finitely many disjoint closed disks. This notion of generalized integral points was introduced in \cite{Phung} and we recall it below.

\begin{defn}
Choose disjoint closed topological disks $\overline\Delta_1,\ldots,\overline\Delta_t$ containing $\Sing(f)$, with distinct points of $\Sing(f)$ in distinct disks, and put
\[
B_0=B\setminus\bigsqcup_{i=1}^t\overline\Delta_i.
\]
For $s\in\mathbb Z_{\ge0}$, define
\[
I(s,B_0)=\{P\in X(K):\#f(\sigma_P(B_0)\cap\mathcal D)\le s\}.
\]
These are the generalized $(s,B_0)$-integral points of $(\mathcal X,\mathcal D)$.
\end{defn}

The definition includes all $(S,\mathcal D)$-integral points with $\#S\le s$:
\[
\bigcup_{\#S\le s}\mathcal X(\mathcal O_{S,\mathcal D})\subseteq I(s,B_0).
\]
Thus $I(s,B_0)$ is potentially much larger than the integral points for any single fixed $S$: it allows the intersection points in $B_0$ to move, and it puts no restriction at all on intersections above the removed disks. Understanding its growth as a function of $s$, with $B_0$ fixed and with the trace quotient taken in the abelian case, is the arithmetic application of the loop-length estimates in this paper.

\subsection{Applications to counting and heights}\label{sec:intro-applications}
Building on Parshin's foundational work~\cite{Parshin},
Phung in~\cite{Phung} proved an important quantitative bound for generalized integral points on abelian varieties modulo the constant trace. We work in his setting, making the geometric interpretation of the divisor hypothesis and the abelian-scheme structure of the model explicit.

\medskip
\noindent\texttt{(P)} Let $A$ be an abelian variety of dimension $n$ over $K=\mathbb C(B)$, let $D\subset A$ be a nonzero reduced effective divisor, and fix a model $(\mathcal A,\mathcal D)$ with morphism $f\colon\mathcal A\to B$. Let $\sigma_O$ be the section of the identity of $A$. We assume:
\begin{enumerate}[(i)]
\item The disks defining $B_0$ contain $\Sing(f)$, with distinct singular-fibre points in distinct disks; over $B\setminus\Sing(f)$ the model is an abelian scheme with zero section $\sigma_O$.
\item $B_0$ is hyperbolic. This is automatic when $t\ge1$.
\item $D_{\overline K}$ contains no translate of a positive-dimensional abelian subvariety of $A_{\overline K}$.
\end{enumerate}
The abelian-scheme structure in (i) is part of the model data used below. Such models can be obtained by compactifying the abelian scheme over the good-reduction locus.

The ampleness assumption appearing in \cite{Phung} is automatic under (iii). Indeed, the identity component of the translation stabilizer of $D_{\overline K}$ would otherwise give positive-dimensional cosets contained in $D_{\overline K}$. Hence that stabilizer is finite and $D$ is ample, by the stabilizer criterion \cite[Proposition~3.4.8]{ConradFeng}. 

Write
\[
T_A=\operatorname{im}\bigl(\Tr_{K/\mathbb C}(A)(\mathbb C)\longrightarrow A(K)\bigr),
\qquad M=A(K)/T_A,
\qquad r=\rk M.
\]
The group $M$ is finitely generated by Lang--N\'eron; see \cite{ConradTrace}. We write $I(s,B_0)/T_A$ for the image of $I(s,B_0)$ in $M$; no invariance of $I(s,B_0)$ under trace translations is asserted. Let $k$ denote the minimal number of generators of $\pi_1(B_0,b_0)$. For $t\ge1$, this group is free and $k=2g(B)+t-1$; for $t=0$, $k=2g(B)$.

\begin{thm}[Phung {\cite[Theorem~A]{Phung}}]\label{thm:main_phung}
In setting \texttt{(P)}, there is a constant $m>0$, depending on the fixed model, divisor and $B_0$, such that
\begin{equation}\label{eq:main_bound_phung}
\#(I(s,B_0)/T_A)\le m(s+1)^{2nk}
\qquad(s\in\mathbb Z_{\ge0}).
\end{equation}
\end{thm}

\noindent Note that the bound behaves coherently: adding further removed disks does not decrease $k$ and shrinks $B_0$, so the set of generalized integral points grows as well.

The proof of Theorem~\ref{thm:main_phung} rests on Phung's implementation of the Parshin cocycle argument. Let $G=\pi_1(B_0,b_0)$ and fix generators $\alpha_1,\ldots,\alpha_k$. By Ehresmann's theorem the family over $B_0$ gives a short exact sequence
\[
1\longrightarrow\Gamma\longrightarrow\pi_1(\mathcal A_{B_0},w_0)
\longrightarrow G\longrightarrow1,
\qquad\Gamma=H_1(\mathcal A_{b_0},\mathbb Z)\simeq\mathbb Z^{2n}.
\]
Each section $\sigma_P$ induces a splitting. Its difference from the splitting defined by $\sigma_O$ is a $1$-cocycle $c_P\colon G\to\Gamma$. The cohomology class determines the image of $P$ modulo the trace up to finitely many possibilities. Thus it suffices to count the possible cocycle classes.

Phung does this by choosing loops representing the generators and estimating the corresponding displacements in the universal cover of $\mathcal A_{B_0}$. After removing the bad fibres of $\mathcal D$, a Brody--Green comparison gives
\[
\ell_h(\sigma_P(\gamma_j))\le c^{-1}\ell_S(\gamma_j),
\]
where $h$ is a fixed Hermitian metric and $S$ is the set of intersections of the section with $\mathcal D$ inside $B_0$. Phung's linear bound $\ell_S(\gamma_j)=O(s+1)$ yields a displacement bound $H(s)=O(s+1)$. A lattice ball in $\Gamma\simeq\mathbb Z^{2n}$ contains $O((H(s)+1)^{2n})$ points. There are $k$ generators, which explains the exponent $2nk$ in \eqref{eq:main_bound_phung}.

This counting method, including the passage from short loops to bounded cocycles, is Phung's. We recall it in detail in \S\ref{sec:phung_proof}, with the metric, base-point and monodromy verifications needed for the argument. Our improvement enters through a sharper estimate for the loops in the base. Substituting $H(s)=O(\sqrt{s+1})$ in the same lattice count gives the following result.
\begin{mainthm}\label{thm:main}
In setting \texttt{(P)}, there is a constant $m>0$, depending on the fixed model, divisor and $B_0$, such that
\begin{equation}\label{eq:main_bound}
\#(I(s,B_0)/T_A)\le m(s+1)^{nk}
\qquad(s\in\mathbb Z_{\ge0}).
\end{equation}
More precisely, after changing $m$ if necessary,
\begin{equation}\label{eq:rank-bound}
\#(I(s,B_0)/T_A)\le m(s+1)^{r/2}.
\end{equation}
\end{mainthm}

The passage from the exponent $nk$ to $r/2$ uses the same radius $O(\sqrt{s+1})$ in a smaller lattice. The first count treats the $k$ cocycle values as a tuple in $\Gamma^k$, a lattice of rank $2nk$. For the refinement, the Parshin--Phung homomorphism from $M=A(K)/T_A$ to the cohomology of the fibre lattice has finite kernel, so it becomes injective after tensoring with $\mathbb R$. The bounded cocycle representatives therefore bound the coordinates of $P$ in $M\otimes\mathbb R$ by $O(\sqrt{s+1})$, by equivalence of norms on fixed finite-dimensional spaces. Counting in the rank-$r$ lattice $M/M_{\mathrm{tors}}$ now gives $O((s+1)^{r/2})$; torsion contributes only a fixed finite factor. In particular, $r\le2nk$. The details are given in \S\ref{sec:MW-height}.

The same method gives a linear height estimate, establishing the fixed-model geometric Lang--Vojta conjecture in this setting.
\begin{maincor}\label{cor:height}
In setting \texttt{(P)}, fix a symmetric ample line bundle $L$ on $A$ and its canonical height $\widehat h_L$. There is a constant $C>0$ such that
\[
\widehat h_L(P)\le C(s+1)\qquad(P\in I(s,B_0)).
\]
If $\mathcal D$ is Cartier, there is also a constant $C'>0$ such that
\[
\deg_B\sigma_P^*\mathcal D\le C'(s+1)\qquad(P\in I(s,B_0)).
\]
All constants refer to the fixed family, divisor, $B_0$ and chosen height data; they are independent of $s$ and $P$.
\end{maincor}

Phung's finiteness theorem also has a height-theoretic consequence. In the ample-divisor setting, \cite[Corollary~1.3]{Phung} gives a finite bound, for each fixed $s$, for the heights of generalized integral points. This relates his result to geometric Lang--Vojta in two ways. The intersection condition is weaker, because it is imposed only over $B_0$, whereas the height conclusion supplies a bound $M(s)$ without a specified linear rate as $s\to\infty$. The issue is therefore the growth of $M(s)$, not merely its dependence on $s$.

Corollary~\ref{cor:height} supplies that linear rate, with the family and base fixed. Its proof uses a norm bound for the image of the Mordell--Weil group in the cohomology of the fibre lattice; the cardinality estimate alone would not imply it. 

More precisely, Corollary~\ref{cor:height} proves Conjecture~\ref{conj:lang-vojta} for abelian varieties over $\mathbb C(B)$ in setting \texttt{(P)}, when $D$ has normal crossings and $\mathcal D$ is Cartier, with $Z=\varnothing$. Indeed, fix one admissible $B_0$ once and for all. For any finite $S\subset B$, every $(S,\mathcal D)$-integral point belongs to $I(s,B_0)$ with $s=\#S$. Since
\[
s+1\le4\max\{1,2g(B)-2+s\}\qquad(g(B),s\in\mathbb Z_{\ge0}),
\]
the corollary gives
\[
\deg_B\sigma_P^*\mathcal D
\le4C'\max\{1,2g(B)-2+\#S\},
\]
which is the conjectured inequality with $m=4C'$, independent of $S$ and $P$. For the numerical inequality, if $s\le3$ then $s+1\le4$, while if $s\ge4$ then $s+1\le4(s-2)\le4(2g(B)-2+s)$. Thus the cases of genus zero or one and small $S$ are included. The dependence on the auxiliary $B_0$ is harmless here, since it has been fixed for the given family. The corollary also applies without the normal-crossings assumption and allows unrestricted intersections above the removed disks.

The exponent in \eqref{eq:main_bound_phung} can also be compared with the classical theory of ordinary integral points. Put
\[
J(s):=\{P\in A(K):\#f(\sigma_P(B)\cap\mathcal D)\le s\}
\subseteq I(s,B_0).
\]
For nonisotrivial elliptic curves, Hindry--Silverman's height estimates and Shioda's rank bound give polynomial counting bounds for ordinary integral points; see \cite[\S2.3]{Phung} and \cite{HS,Shioda}. This is one motivation for improving the exponent for the larger set $I(s,B_0)$. The trace qualification is essential: on a constant elliptic curve with its constant zero divisor there are infinitely many constant integral sections. Our statements consistently count the image in $A(K)/T_A$.

In \S\ref{sec:phung_proof} we explain Phung's six-step proof in detail. Replacing his linear loop bound with Proposition~\ref{prop:common-base} changes the lattice radius from $O(s+1)$ to $O(\sqrt{s+1})$ and halves the exponent. Finally, the same bounded cocycle representatives control coordinates in the Mordell--Weil lattice. Counting in its actual rank $r$ proves \eqref{eq:rank-bound}; the quadratic nature of the canonical height proves Corollary~\ref{cor:height}.

\paragraph{Funding.}
The author is supported by the NSFC RFIS-I 2025, project titled \emph{``Interdisciplinary Perspectives in Diophantine Geometry: Analytic, Dynamical and Arithmetic Explorations''}, and by the Zhejiang Science and Technology Development Fund [2024] No.\ 28, project titled \emph{``New Perspectives in Diophantine Geometry''}.

\section{Kobayashi length estimates}\label{sec:extremal}

\subsection{Kobayashi pseudo-metric}\label{sec:kob_metric}
Let $X$ be a connected complex manifold with complex structure $J$.
Throughout, $TX$ denotes the real tangent bundle $T^{\mathbb R}X$,
equipped with the complex scalar multiplication
\[
\lambda v=(\operatorname{Re}\lambda)v
          +(\operatorname{Im}\lambda)J_xv,
\qquad x\in X,\quad v\in T_x^{\mathbb R}X,\quad \lambda\in\mathbb C.
\]
It has complex rank $\dim_{\mathbb C}X$ and is canonically identified
with the holomorphic tangent bundle  $T^{1,0}X$ by
\[
v\longmapsto \frac12(v-iJ_xv).
\]
We use these two descriptions interchangeably. This should be
distinguished from the complexification
\[
T^{\mathbb R}X\otimes_{\mathbb R}\mathbb C
=T^{1,0}X\oplus T^{0,1}X,
\]
whose complex rank is $2\dim_{\mathbb C}X$.
\begin{defn}
A \emph{Finsler pseudo-metric} on $X$ is a function
\[
\begin{aligned}
F\colon TX &\to \mathbb R_{\ge 0}\\
(x,v)&\mapsto F(x,v)
\end{aligned}
\]
satisfying the homogeneity condition $F(x,\lambda v)=|\lambda|\,F(x,v)$ for all $\lambda\in\mathbb C$. If $F(x,v)>0$ for all $x\in X$ and $v\in T_{X,x}\setminus\{0\}$, then we say that $F$ is a \emph{Finsler metric}.
\end{defn}

Throughout, by a \emph{path} we mean a piecewise $C^1$ map $\gamma\colon[0,1]\to X$; this entails no loss of generality for the distances defined below. Given a Borel measurable Finsler pseudo-metric $F$ and a path $\gamma$, the function $t\mapsto F(\gamma(t),\gamma'(t))$ is measurable, and we define the length of $\gamma$ as
\[
\ell_F(\gamma):=\int_0^1 F(\gamma(t),\gamma'(t))\,dt\ \in[0,+\infty].
\]
When  $F$ is \emph{upper semicontinuous} we have $\ell_F(\gamma)<+\infty$, since $F$ is then locally bounded above and $\gamma([0,1])$ is compact. Given a set of paths $\Gamma$ we set
\[
\ell_F(\Gamma):=\inf_{\gamma\in\Gamma}\ell_F(\gamma)\,.
\]
We endow $X$ with the structure of extended pseudo-metric space by setting
\[
d_F(x,y):=\inf_\gamma \ell_F(\gamma)\,,\quad \text{where the infimum runs over all paths $\gamma$ joining $x$ and $y$}\,.
\]
Moreover, if $Y,Z$ are two subsets of $X$, we put
\[
\diam_F(Y):=\sup_{y,y'\in Y}d_F(y,y')\,,\quad d_F(Y,Z):=\inf_{y\in Y,\,z\in Z}d_F(y,z)\,.
\]
One can show that a \emph{continuous} Finsler metric induces a metric $d_F$. Any Hermitian metric $h$ on $X$ induces a Finsler metric $F_h(x,v)=\sqrt{h_x(v,v)}$, hence a distance which we denote by $d_h$.

  An important example of a Finsler pseudo-metric is the \emph{Kobayashi--Royden pseudo-metric}
\[
\kappa_X(x,v):=\inf\left\{\frac{2}{R}\in\mathbb R_{>0}\colon \exists\, f\in\mathrm{Hol}(\Delta(0,R), X)\text{ such that } f(0)=x,\ f'(0)=v\right\}\,,
\]
where $\mathrm{Hol}(\Delta(0,R), X)$ is the set of holomorphic maps from the open disk of radius $R$ centered at $0$ to $X$. Royden \cite[Proposition 3]{Roy} shows that $\kappa_X$ is upper semicontinuous, hence Borel measurable. Since $\kappa_X$ is intrinsic to $X$, the induced topological pseudo-metric is denoted by $d_X$ and called the \emph{Kobayashi pseudo-metric}. We say that $X$ is \emph{Kobayashi hyperbolic} if $d_X$ is a metric.

If $X$ carries a Hermitian metric $h$ whose holomorphic sectional curvature is bounded above by a negative constant, then by the Ahlfors--Schwarz lemma one has $\kappa_X\ge c\,\|\cdot\|_h$ for some constant $c>0$; in particular $d_X\ge c\,d_h$ and $X$ is Kobayashi hyperbolic. In dimensions at least two the converse is false in general: Demailly constructed a smooth projective Kobayashi hyperbolic surface which admits no Hermitian metric of negative holomorphic sectional curvature \cite[Theorem~8.2]{Demailly}. On a Riemann surface, however, the following holds: by the uniformization theorem if $X$ is Kobayashi hyperbolic its universal cover is the disk; so $X$ is biholomorphic to $\Gamma\backslash\mathbb H$ for a torsion-free Fuchsian group $\Gamma\cong\pi_1(X)$ (no elliptic elements, parabolics allowed). The curvature $-1$ Poincar\'e metric $ds^2=(dx^2+dy^2)/y^2$ on $\mathbb H$ descends to $X$, and since holomorphic covering maps are infinitesimal isometries for $\kappa_X$, the descended metric \emph{is} the Kobayashi metric. Thus on Riemann surfaces ``admits a Hermitian metric of holomorphic sectional curvature bounded above by a negative constant'' and ``Kobayashi hyperbolic'' coincide.

From the definition one deduces the distance-decreasing property of the Kobayashi--Royden metric: if $\phi\colon X\to Y$ is holomorphic, then
\[
\kappa_Y(\phi(x), d_x\phi(v))\le\kappa_X(x,v)\,,\quad\forall (x,v)\in TX\,.
\]
This in turn implies the distance-decreasing property of the Kobayashi distance
\[
d_Y(\phi(x),\phi(x'))\le d_X(x,x')\,,\quad\forall x,x'\in X\,.
\]
\subsection{Growth of loop lengths in terms of the cusps}\label{sec:growth}
We fix a compact Riemann surface $B$ of genus $g$ endowed with a Riemannian metric $\rho$. Let $\overline\Delta_1,\ldots,\overline{\Delta}_t$ be disjoint disks on $B$ and define $B_0:=B\smallsetminus\sqcup^t_{i=1}\overline\Delta_i$. We initially assume that $t\ge1$, so that $B_0$ is hyperbolic; the case $t=0$ is treated at the end of \S\ref{sec:reduction-to-simple}. Let $S\subseteq B_0$ be a finite set (possibly empty), put $s:=\#S$, and let $\kappa_S$ be the Kobayashi--Royden metric of the hyperbolic surface $B_0\smallsetminus S$, with induced length and distance $\ell_S$, $d_S$. We define the \emph{distortion function}
\begin{equation}\label{eq:dist_fun}
\lambda_S(b):=\sup_{v\in T_bB\smallsetminus\{0\}}\frac{\kappa_S(b,v)}{\sqrt{\rho_b(v,v)}}\,,\qquad b\in B_0\smallsetminus S\,,
\end{equation}
where $T_bB=T^{\mathbb R}_bB$ is the \emph{real} tangent plane (real dimension $2$), on which both $\kappa_S(b,\cdot)$ and $\sqrt{\rho_b(\cdot,\cdot)}$ are defined. The ratio is homogeneous of degree $0$, so the supremum may be taken over the compact $\rho$-unit circle. On this circle the denominator equals $1$, and the numerator $\kappa_S(b,\cdot)$ is upper semicontinuous. Thus $\lambda_S(b)$ is finite and $\lambda_S$ is Borel on $B_0\smallsetminus S$. In a holomorphic coordinate $z=x+iy$, identifying $a\partial_x+b'\partial_y\leftrightarrow(a+ib')\partial_z$, $\mathbb C$-homogeneity gives $\kappa_S(b,a\partial_x+b'\partial_y)=\kappa_S(b,\partial_z)\sqrt{a^2+b'^2}$, while $\rho_b$ has eigenvalues $\mu_1(b)\le\mu_2(b)$ in the frame $(\partial_x,\partial_y)$. Maximizing the ratio over the Euclidean unit circle thus minimizes $\rho_b(v,v)$, giving
\[
\lambda_S(b)=\frac{\kappa_S(b,\partial_z)}{\sqrt{\mu_1(b)}}\,.
\]
If $\rho$ is conformal ($\mu_1=\mu_2$, or equivalently $\rho$ is hermitian) the ratio  of \Cref{eq:dist_fun} is independent of $v$ (so that the supremum on $v$ is superfluous) and $\lambda_S$ is smooth; we do not assume this.

\smallskip
\noindent\textbf{Standing assumption for \S\ref{sec:growth} (lifted in \S\ref{sec:reduction-to-simple}).} Throughout this subsection, we fix a loop $\alpha$ in $B_0$ that doesn't represent a trivial class in $\pi_1(B_0,b_0)$ \emph{and that admits a smooth simple representative} $\gamma\Subset B_0$ freely homotopic to $\alpha$. Here and throughout, a \emph{smooth simple loop} is an injective smooth immersion $\gamma\colon S^1\to B_0$; since $S^1$ is compact, $\gamma$ is then a smooth embedding, and its image is a smoothly embedded circle. In particular $\gamma'$ never vanishes, so $\gamma$ may be parametrized by $\rho$-arc length, with $\gamma'(0)=\gamma'(\ell_\rho(\gamma))$. Such loops exist: the basis loops of $\pi_1(B_0,b_0)$ described in \S\ref{sec:reduction-to-simple} are smooth simple, and the general case (no simplicity hypothesis) is reduced to this one in \S\ref{sec:reduction-to-simple}.

\smallskip
\noindent Parametrize $\gamma$ by its $\rho$-arc-length and let $\eta(\tau)$ be the unit normal vector of $\gamma$ at $\tau$; here and throughout we fix the convention that $\eta(\tau)$ is the rotation of the unit tangent vector $\gamma'(\tau)$ by $+\pi/2$, with respect to $\rho$ and the orientation of the Riemann surface~$B$. Since $\gamma$ is smooth and simple and compactly contained in~$B_0$, for $\delta\in\mathbb R_{>0}$ sufficiently small the exponential map from the normal bundle is a diffeomorphism, yielding:
\[
\begin{aligned}
\mathbb R/\ell_\rho(\gamma)\mathbb Z\times[-\delta,
\delta]&\to  \{b\in B_0\colon d_\rho(b,\gamma)\le\delta\}=:T_\delta(\gamma)\Subset B_0\\
(\tau,u) &\mapsto\exp_{\gamma(\tau)}(u\eta(\tau))
\end{aligned}
\]
The coordinates $(\tau,u)$ are called \emph{Fermi  coordinates}, and $T_{\delta}(\gamma)$ is the \emph{Fermi strip} with respect to  $\gamma$. For any $u\in[-\delta,\delta]$ we define the \emph{parallel curve} $\gamma_u\colon[0,\ell_\rho(\gamma)]\to B_0$, $\ \gamma_u(\tau):=\exp_{\gamma(\tau)}(u\eta(\tau))$.

\begin{figure}[ht]
\centering
\begin{tikzpicture}[scale=1.2,>=Stealth,
  puncture/.style={black, line width=.75pt},
  guide/.style={gray!65, thin},
  every node/.style={font=\small}]
\fill[blue!7, even odd rule] (0,0) circle (3) (0,0) circle (1.7);
\draw[blue!50!black, line width=.7pt] (0,0) circle (3);
\draw[blue!50!black, line width=.7pt] (0,0) circle (1.7);

\draw[black, line width=1.05pt] (0,0) circle (2.35);
\draw[black, line width=1.05pt, ->]
  (35:2.35) arc[start angle=35,end angle=55,radius=2.35];
\draw[red!75!black, dashed, line width=.85pt] (0,0) circle (2.02);
\draw[red!75!black, line width=.85pt, ->]
  (35:2.02) arc[start angle=35,end angle=55,radius=2.02];

\draw[guide] (145:2.35) -- (-3.05,1.65);
\node[anchor=east] at (-3.08,1.65) {$\gamma=\gamma_0$};
\draw[red!75!black, thin] (215:2.02) -- (-1.05,-.70);
\node[red!75!black, anchor=west] at (-1.05,-.70) {$\gamma_u$};
\draw[guide] (25:3) -- (3.25,1.55);
\node[blue!50!black, anchor=west] at (3.28,1.55) {$\gamma_{-\delta}$};
\draw[guide] (270:1.7) -- (0,-1.20);
\node[blue!50!black, anchor=south] at (0,-1.20) {$\gamma_{\delta}$};

\fill (0,2.35) circle (1.4pt);
\draw[gray!80, line width=.8pt, ->] (0,2.35) -- (0,1.48);
\node[gray!80, anchor=north west] at (.10,1.48) {$\eta(\tau_0)$};
\draw[gray!80, thin] (0,2.23) -- (-.12,2.23) -- (-.12,2.35);

\draw[gray!45, thin] (2.02,0) -- (2.02,-.30);
\draw[gray!45, thin] (2.35,0) -- (2.35,-.30);
\draw[red!75!black, thin, {Stealth[length=3pt]}-{Stealth[length=3pt]}]
  (2.02,-.24) -- (2.35,-.24)
  node[midway, below=3pt, inner sep=0pt] {$u$};

\draw[guide] (0,-1.7) -- (.40,-1.7);
\draw[guide] (0,-3) -- (.40,-3);
\draw[gray!80, thin, <->] (.32,-1.7) -- (.32,-3);
\node[gray!80, anchor=west] at (.43,-2.77) {$2\delta$};

\foreach \ang/\rad/\lab/\pos in
  {110/2.70/1/above,170/2.70/2/above,
   225/2.68/3/below,305/2.60/4/right,350/2.70/5/above}{
  \begin{scope}[shift={(\ang:\rad)}]
    \draw[puncture] (-.045,-.045)--(.045,.045)
                    (-.045,.045)--(.045,-.045);
    \node[\pos, inner sep=3pt, font=\scriptsize] {$p_{\lab}$};
  \end{scope}
}
\node[blue!50!black] at (0,.45) {$T_\delta(\gamma)\cong S^1\times[-\delta,\delta]$};
\node[gray!75, font=\scriptsize] at (0,.06) {$0<u<\delta$};
\end{tikzpicture}
\caption{The Fermi strip drawn schematically as an annulus. The central
loop $\gamma$ and the parallel curve $\gamma_u$ are closed; positive $u$
points along the chosen normal. The punctures $p_1,\ldots,p_5\in S$
are avoided by the displayed $\gamma_u$.}
\label{fig:fermi}
\end{figure}

Since $\gamma$ is a hypersurface in $B$, its Fermi coordinates
are semigeodesic
\cite[Example~6.43 and Corollary~6.42(a)]{LeeRiem}.
This means that the $u$-coordinate curves are unit-speed
geodesics orthogonal to the curves $u=\mathrm{constant}$.
Consequently,
\[
\rho_{uu}=1,\qquad \rho_{\tau u}=0,
\qquad
\rho=\rho_{\tau\tau}(\tau,u)\,d\tau^2+du^2.
\]
The area density is the square root of the determinant of the
metric matrix. Since
\[
\det
\begin{pmatrix}
\rho_{\tau\tau} & 0\\
0 & 1
\end{pmatrix}
=\rho_{\tau\tau},
\]
we obtain
\begin{equation}\label{eq:fermi_coor_int}
dA_\rho=\sqrt{\rho_{\tau\tau}(\tau,u)}\,d\tau\,du.
\end{equation}
Along a fixed parallel curve $\gamma_u$, we have $du=0$, so its
length element is
$ds_\rho=\sqrt{\rho_{\tau\tau}(\tau,u)}\,d\tau$.
Thus, for every nonnegative Borel function $f$ on $T_\delta(\gamma)$,
\begin{equation}\label{eq:fermi-averaging}
\begin{aligned}
\int_{T_\delta(\gamma)} f\,dA_\rho
&=\int_{-\delta}^{\delta}
\left(
\int_0^{\ell_\rho(\gamma)}
f(\gamma_u(\tau))
\sqrt{\rho_{\tau\tau}(\tau,u)}\,d\tau
\right)du\\
&=\int_{-\delta}^{\delta}
\left(\int_{\gamma_u}f\,ds_\rho\right)du.
\end{aligned}
\end{equation}
The first equality is the change-of-variables formula for
$(\tau,u)\mapsto\gamma_u(\tau)$; the second uses the length
element along $\gamma_u$.

\begin{lem}\label{lem:main_lem}
Let $T\Subset B_0$ be a fixed compact set. There exists a constant $A_T>0$, independent of the finite set $S\subset B_0$, such that
\[
\int_T\lambda_S\,dA_\rho\le A_T\sqrt{\#S+1}.
\]
The values of $\lambda_S$ at points of $S$ may be defined arbitrarily for this integral.
\end{lem}
\begin{proof}
Put $s=\#S$ and $a=\operatorname{Area}_\rho(T)$, and set $\lambda_S=0$ at the points of $S$. If $a=0$ there is nothing to prove. We treat $s\ge1$ first and address $s=0$ at the end. We divide the argument into three steps, keeping track of why the constants are independent of the positions of the punctures.

\medskip\noindent\emph{Step 1: A local pointwise bound.}
There are constants $c_1,c_2,r_1,r_2>0$, depending only on $B$ and $\rho$, such that
\[
\kappa_{\Delta_\rho(b,r)}(b,v)\le\frac{c_1}{r}|v|_\rho
\quad(0<r<r_1),\qquad
\operatorname{Area}_\rho(\Delta_\rho(b,r))\le c_2r^2
\quad(0<r\le r_2).
\]
These are the local metric and area estimates of \cite[Lemmas~3.3 and~3.5]{Phung}. We recall why their constants can be chosen uniformly, including when $\rho$ is not conformal. Choose finitely many holomorphic coordinate charts $z_i\colon U_i\to\mathbb C$, extending to neighborhoods of $\overline{U_i}$, and open sets $V_i\Subset U_i$ covering $B$. Smoothness and positive definiteness of $\rho$ give constants $0<\mu\le\nu$ such that
\[
\mu\,|dz_i(v)|\le |v|_\rho\le\nu\,|dz_i(v)|
\qquad\text{on every }U_i.
\]
Compactness of the finitely many $\overline{V_i}$ allows us to choose $r_0>0$ such that, for $b\in V_i$ and $0<r<r_0$, the metric ball $\Delta_\rho(b,r)$ stays in $U_i$, and the Euclidean disk of radius $r/\mu$ centered at $z_i(b)$ stays in $z_i(U_i)$. In these coordinates,
\[
\Delta(z_i(b),r/\nu)\subseteq z_i(\Delta_\rho(b,r))
\subseteq\Delta(z_i(b),r/\mu).
\]
The first inclusion follows by estimating the $\rho$-length of a straight coordinate segment. For the second, a minimizing $\rho$-geodesic from $b$ to a point of the metric ball stays in that ball, hence in $U_i$, and its Euclidean coordinate length is at most its $\rho$-length divided by $\mu$.
The Kobayashi metric at the centre of a Euclidean disk of radius $R$ is $2|\cdot|/R$, with our normalization. Holomorphic contraction for the first inclusion therefore gives
\[
\kappa_{\Delta_\rho(b,r)}(b,v)
\le\frac{2\nu}{r}|dz_i(v)|
\le\frac{2\nu}{\mu r}|v|_\rho.
\]
Also $dA_\rho\le\nu^2\,dx\,dy$ in the chart, so the second inclusion gives
\[
\operatorname{Area}_\rho(\Delta_\rho(b,r))
\le\pi\frac{\nu^2}{\mu^2}r^2.
\]
Thus one may take $c_1=2\nu/\mu$, $c_2=\pi\nu^2/\mu^2$, $r_1=r_0$ and $r_2=r_0/2$. These choices use only the fixed surface and background metric, and impose no separation condition on the punctures. Now choose
\[
d_*:=\tfrac12\min\{d_\rho(T,\partial B_0),r_1,r_2\}>0.
\]
For $b\in T\setminus S$, set
\[
r:=\min\{d_*,d_\rho(b,S)\}.
\]
This choice ensures that the open metric ball
$\Delta_\rho(b,r)$ stays inside $B_0$ and avoids $S$.
Indeed, $r\le d_*<d_\rho(T,\partial B_0)$ prevents the ball
from reaching $\partial B_0$, while $r\le d_\rho(b,S)$
excludes every puncture. A puncture at distance exactly $r$
is not in the open ball. Moreover, $0<r\le d_*\le r_1/2<r_1$,
so the local metric estimate applies. The holomorphic inclusion
$\Delta_\rho(b,r)\hookrightarrow B_0\setminus S$ and that
estimate give, for every $v\in T_bB$,
\[
\kappa_S(b,v)
\le\kappa_{\Delta_\rho(b,r)}(b,v)
\le\frac{c_1}{r}|v|_\rho.
\]
Dividing by $|v|_\rho$ for $v\ne0$ and taking the supremum yields
\begin{equation}\label{eq:point_bound}
\lambda_S(b)\le
\frac{c_1}{\min\{d_*,d_\rho(b,S)\}}.
\end{equation}
For $S=\varnothing$, we use $d_\rho(b,S)=+\infty$;
then $r=d_*$ and the same argument gives
$\lambda_\varnothing(b)\le c_1/d_*$ on $T$.

\medskip\noindent\emph{Step 2: The area of the high-distortion region.}
For $u>c_1/d_*$ and $b\in T\setminus S$, \eqref{eq:point_bound} gives
\[
\lambda_S(b)>u
\;\Longrightarrow\;
\min\{d_*,d_\rho(b,S)\}<c_1/u<d_*
\;\Longrightarrow\;
d_\rho(b,S)<c_1/u.
\]
Since $\lambda_S=0$ on $S$, it follows that
\[
E_u:=\{b\in T:\lambda_S(b)>u\}
\subseteq T\cap\bigcup_{p\in S}\Delta_\rho(p,c_1/u).
\]
The area estimate from Step~1,
\[
\operatorname{Area}_\rho\Delta_\rho(p,r)\le c_2r^2
\qquad(p\in B,\;0<r\le r_2),
\]
applies with $r=c_1/u<d_*\le r_2/2$. Since $E_u\subseteq T$, we have
$\operatorname{Area}_\rho(E_u)\le a$.
On the other hand, the covering by metric balls and subadditivity give
\[
\operatorname{Area}_\rho(E_u)
\le\sum_{p\in S}\operatorname{Area}_\rho\Delta_\rho(p,c_1/u)
\le c_2c_1^2s/u^2,
\]
where $\#S=s$; the balls need not be disjoint.
Combining these two upper bounds yields
\begin{equation}\label{eq:distribution}
\operatorname{Area}_\rho(E_u)
\le\min\{a,c_2c_1^2s/u^2\}.
\end{equation}

\medskip\noindent\emph{Step 3: Integrating the distribution function.}
For $s\ge1$, set $u_*=c_1\sqrt{c_2s/a}$. By the layer-cake formula
\cite[Theorem~1.13]{LiebLoss}, the bound
$\operatorname{Area}_\rho\{b\in T:\lambda_S(b)>u\}\le a$,
and \eqref{eq:distribution},
\begin{align*}
\int_T\lambda_S\,dA_\rho
&=\int_0^\infty
  \operatorname{Area}_\rho(E_u)\,du\\
&\le\int_0^{c_1/d_*}a\,du
  +\int_{c_1/d_*}^\infty\min\{a,c_2c_1^2s/u^2\}\,du\\
&\le\frac{c_1a}{d_*}
  +\int_0^\infty\min\{a,c_2c_1^2s/u^2\}\,du\\
&=\frac{c_1a}{d_*}
  +\int_0^{u_*}a\,du
  +c_2c_1^2s\int_{u_*}^\infty u^{-2}\,du\\
&=\frac{c_1a}{d_*}+au_*+\frac{c_2c_1^2s}{u_*}\\
&=\frac{c_1a}{d_*}+2c_1\sqrt{c_2as}
 \le A_T\sqrt{s+1},
\end{align*}
where
\[
A_T=\frac{c_1a}{d_*}+2c_1\sqrt{c_2a}
\]
is independent of $S$. For $s=0$, Step~1 gives
$\int_T\lambda_\varnothing\,dA_\rho\le c_1a/d_*\le A_T$.
\end{proof}

\begin{defn}\label{def:L_alpha_s}
For a nontrivial loop $\alpha$ in $B_0$, let $\Gamma_S^\alpha$
be the set of loops in $B_0\setminus S$ freely homotopic to $\alpha$ in $B_0$, and set
\[
L(\alpha,s)=\sup_{\substack{S\subset B_0\\\#S\le s}}
\inf_{\gamma\in\Gamma_S^\alpha}\ell_S(\gamma).
\]
\end{defn}
For the bordered surfaces considered here, nontrivial free homotopy classes have positive length in $B_0$; see \cite[Proposition~2.24]{bor}. Distance-decreasing therefore gives a positive lower bound for the infimum defining $L(\alpha,s)$.

\paragraph{Proof of Theorem~\ref{thm:new_bound_L} for simple loops.}
Choose a smooth simple representative $\gamma\Subset B_0$ of $\alpha$
and a Fermi strip $T=T_\delta(\gamma)\Subset B_0$, with coordinates
defined for $|u|\le\delta$. Set $I=(-\delta,\delta)$ and
\[
E=\{u\in I:\gamma_u\cap S\ne\varnothing\}.
\]
Injectivity of the strip coordinates gives $\#E\le\#S$. Recall that
in Fermi coordinates
\[
\gamma_u'(\tau)=\partial_\tau|_{(\tau,u)},\qquad
\rho_{\tau\tau}(\tau,u)
=\rho_{\gamma_u(\tau)}(\gamma_u'(\tau),\gamma_u'(\tau)).
\]
For $u\in I\setminus E$, the definition of $\lambda_S$ as a directional
supremum therefore gives
\begin{align*}
\kappa_S(\gamma_u(\tau),\gamma_u'(\tau))
&\le\lambda_S(\gamma_u(\tau))
 \sqrt{\rho_{\gamma_u(\tau)}(\gamma_u'(\tau),\gamma_u'(\tau))}\\
&=\lambda_S(\gamma_u(\tau))\sqrt{\rho_{\tau\tau}(\tau,u)}.
\end{align*}
By the definition of Kobayashi length,
\begin{align*}
\ell_S(\gamma_u)
&=\int_0^{\ell_\rho(\gamma)}
  \kappa_S(\gamma_u(\tau),\gamma_u'(\tau))\,d\tau\\
&\le\int_0^{\ell_\rho(\gamma)}
  \lambda_S(\gamma_u(\tau))
  \sqrt{\rho_{\tau\tau}(\tau,u)}\,d\tau.
\end{align*}
The length function is continuous on $I\setminus E$, since locally the
parallel curves vary smoothly in a compact subset of $B_0\setminus S$.
Extend it by zero on the finite set $E$; the resulting function is measurable.
Integrating the preceding inequality in $u$ and using
\eqref{eq:fermi-averaging} with $f=\lambda_S$ extended by zero on $S$,
we obtain the equality below; the final inequality follows from
Lemma~\ref{lem:main_lem}:
\begin{equation}\label{eq:avg_length}
\begin{aligned}
\int_I\ell_S(\gamma_u)\,du
&\le\int_I\int_0^{\ell_\rho(\gamma)}
  \lambda_S(\gamma_u(\tau))
  \sqrt{\rho_{\tau\tau}(\tau,u)}\,d\tau\,du\\
&=\int_T\lambda_S\,dA_\rho
 \le A_T\sqrt{\#S+1}.
\end{aligned}
\end{equation}

Set $M=A_T\sqrt{\#S+1}/(2\delta)$. We claim that
$\ell_S(\gamma_{u_0})\le M$ for some $u_0\in I\setminus E$.
Otherwise, the measurable sets
\[
F_m=\{u\in I\setminus E:\ell_S(\gamma_u)\ge M+1/m\},
\qquad m\ge1,
\]
satisfy $\bigcup_{m\ge1}F_m=I\setminus E$. Since $|I\setminus E|=2\delta$,
some $F_m$ has positive Lebesgue measure. Hence
\[
\int_I\ell_S(\gamma_u)\,du
\ge M|I\setminus E|+\frac{|F_m|}{m}
>2\delta M=A_T\sqrt{\#S+1},
\]
contrary to \eqref{eq:avg_length}.

The family $\gamma_{tu_0}$, $0\le t\le1$, gives a free homotopy in
$B_0$ from $\gamma$ to $\gamma_{u_0}$. This homotopy may cross $S$,
whereas $\gamma_{u_0}$ avoids $S$; thus $\gamma_{u_0}\in\Gamma_S^\alpha$.
Consequently, for every $S$ with $\#S\le s$,
\[
\inf_{\eta\in\Gamma_S^\alpha}\ell_S(\eta)
\le\frac{A_T}{2\delta}\sqrt{\#S+1}
\le\frac{A_T}{2\delta}\sqrt{s+1}.
\]
The constants are independent of the cardinality and positions of $S$.
Taking the supremum over all such $S$, including $S=\varnothing$, gives
$L(\alpha,s)\le(A_T/2\delta)\sqrt{s+1}$.
Moreover, all selected loops lie in the fixed compact strip $T$, and
\[
L_0:=\max_{|u|\le\delta}\ell_\rho(\gamma_u)<\infty
\]
by smooth dependence on $u$ and compactness of $[-\delta,\delta]$.
\qed

For the application to generalized integral points, the Parshin cocycle
argument requires loops representing the generators of
$\pi_1(B_0,b_0)$ that share a \emph{common base point}
and are conjugated to the generators by a \emph{single}
path.
Theorem~\ref{thm:new_bound_L}, applied independently to
each generator, produces loops with the square-root Kobayashi length bound but with potentially different base points.
The next proposition shows that a common base point can be arranged without degrading the asymptotic bound.

We first record the local tangent adjustment that makes this possible.

\begin{lem}\label{lem:prescribed-tangent}
Let $M$ be a smooth orientable Riemannian surface without boundary, let $p\in M$,
and let $w\in T_pM$ be a unit tangent vector.  Then for every smooth
simple loop $\gamma$ at~$p$ with $\gamma\Subset M$, there exists a smooth
simple loop $\tilde\gamma$ at~$p$ with $\tilde\gamma\Subset M$,
homotopic to~$\gamma$ relative to~$p$, and satisfying
\[
  \frac{\tilde\gamma'(0)}{|\tilde\gamma'(0)|} = w\,.
\]
\end{lem}
\begin{proof}
Let $w_0:=\gamma'(0)/|\gamma'(0)|$ be the unit tangent vector of~$\gamma$
at~$p$.  If $w_0=w$ there is nothing to prove, so assume $w_0\neq w$.
Let $\theta\in\,]0,2\pi[$ be the angle from~$w_0$ to~$w$ and for each
$t\in[0,1]$ let $R_t\in\mathrm{SO}(2)$ denote the rotation of~$T_pM$ by
angle~$t\theta$, so that $R_0=\mathrm{Id}$ and $R_1(w_0)=w$.

Choose $\epsilon>0$ small enough that the exponential map $\exp_p$ is a
diffeomorphism from $B(0,\epsilon)\subset T_pM$ onto an open neighborhood
$U$ of~$p$ in~$M$ with $U\Subset M$.  Choose a smooth cutoff function
$\chi\colon[0,\infty)\to[0,1]$ with $\chi\equiv 1$ on $[0,\epsilon/3]$
and $\chi\equiv 0$ on $[2\epsilon/3,\infty)$.  Define the map
$\phi\colon M\to M$ by
\[
  \phi(x)\;:=\;
  \begin{cases}
    \exp_p\!\bigl(\,R_{\chi(|y|)}\,y\,\bigr),
      & \text{if } x=\exp_p(y)\in U,\\[4pt]
    x, & \text{if } x\notin U,
  \end{cases}
\]
where $|y|$ denotes the norm in~$T_pM$.  Since $\chi\equiv 1$ near
$|y|=0$, the map~$\phi$ equals $\exp_p\circ\,R_1\circ\exp_p^{-1}$ in a
neighborhood of~$p$, and since $\chi\equiv 0$ for $|y|\ge 2\epsilon/3$,
$\phi$~is the identity outside~$U$.  In geodesic polar coordinates
$(r,\psi)$ centered at~$p$ (via $\exp_p$), the map reads
$\phi(r,\psi)=(r,\,\psi+\chi(r)\theta)$: it preserves the radius~$r$ and
rotates each geodesic circle $\{r=\mathrm{const}\}$ by the angle
$\chi(r)\theta$.  On each circle this is a bijection, and distinct circles
map to distinct circles, so $\phi$ is a bijection of~$U$ fixing~$p$;
together with $\phi=\mathrm{id}$ outside~$U$ this makes $\phi$ a bijection
of~$M$.  It is smooth with smooth inverse $\phi^{-1}(r,\psi)=(r,\psi-\chi(r)\theta)$
(both expressions are smooth across $r=0$ because $\chi\equiv 1$ there, where
$\phi$ is the fixed rotation $R_1$), so $\phi$ is a diffeomorphism of~$M$.
Moreover:
\begin{enumerate}[\rm(i)]
  \item $\phi(p)=p$ and $d_p\phi=R_1$, so
    $(\phi\circ\gamma)'(0)=R_1(\gamma'(0))=|\gamma'(0)|\,w$,
    which has unit tangent direction~$w$;
  \item $\phi$ is the identity outside~$U$, so $\phi\circ\gamma$ coincides
    with~$\gamma$ outside~$U$ and in particular
    $\phi\circ\gamma\Subset M$;
  \item $\phi$ is a diffeomorphism, so $\phi\circ\gamma$ is simple;
    \item the family $\phi_t$ (defined by replacing $\chi$ with $t\chi$
    for $t\in[0,1]$) satisfies $\phi_0=\mathrm{Id}$ and $\phi_1=\phi$,
    and $\phi_t(p)=p$ for all~$t$. Hence $t\mapsto\phi_t\circ\gamma$
    is a homotopy of loops based at~$p$ from~$\gamma$
    to~$\tilde\gamma$, so
    $[\tilde\gamma]=[\gamma]$ in $\pi_1(M,p)$.
\end{enumerate}
Setting $\tilde\gamma:=\phi\circ\gamma$ completes the proof.
\end{proof}

\begin{prop}\label{prop:common-base}
Assume $t\ge1$ and fix a free basis $\alpha_1,\ldots,\alpha_k$ of $\pi_1(B_0,b_0)$ whose elements have smooth simple representatives, with $k\ge1$. There are constants $C,\delta',L_0>0$ and a compact set $\mathcal K_0\Subset B_0$ such that for every finite $S\subset B_0$ there are a point $q\in B_0\setminus S$, a path $\sigma$ from $b_0$ to $q$ in $\mathcal K_0$, and loops $\widehat\gamma_j\subset\mathcal K_0\setminus S$ based at $q$, satisfying
\[
[\sigma^{-1}\circ\widehat\gamma_j\circ\sigma]=\alpha_j,
\qquad \ell_S(\widehat\gamma_j)\le C\sqrt{\#S+1},
\]
and
\[
\ell_\rho(\sigma)\le\delta',\qquad
\ell_\rho(\widehat\gamma_j)\le L_0.
\]
The constants and $\mathcal K_0$ are independent of $S$.
\end{prop}
\begin{proof}
Fix a unit tangent vector $w\in T_{b_0}B_0$ and let $\eta_0$ be the
rotation of~$w$ by $+\pi/2$ with respect to~$\rho$ and the orientation
of~$B$.  By
Lemma~\ref{lem:prescribed-tangent}, for each $j=1,\dots,k$ we can choose
a smooth simple loop $\gamma^{(j)}$ at~$b_0$, compactly contained
in~$B_0$, homotopic to~$\alpha_j$ relative to~$b_0$, and whose unit tangent vector
at~$b_0$ equals~$w$.  Since all loops share the tangent direction~$w$
at~$b_0$, and since the unit normal field along every loop is the
$+\pi/2$-rotation of the unit tangent (the convention fixed
in~\S\ref{sec:growth}), their unit normal vectors at~$b_0$ all
equal~$\eta_0$.  For
each~$j$, construct the Fermi strip $T_\delta(\gamma^{(j)})$ as in the
proof of Theorem~\ref{thm:new_bound_L}, with a common width $\delta>0$
small enough for all~$k$ strips.  The parallel curve
$\gamma_u^{(j)}$ has base point
\[
  q(u)\;:=\;\exp_{b_0}(u\,\eta_0),
\]
which is independent of~$j$.

The strips need not be disjoint from one another; only the coordinate map on each individual strip must be injective. For each $j$, let
\[
E_j=\{u\in(-\delta,\delta):\gamma_u^{(j)}\cap S\ne\varnothing\},
\qquad E=\bigcup_{j=1}^kE_j.
\]
Each $E_j$ has at most $\#S$ elements, so $E$ is finite. Define $g_j(u)=\ell_S(\gamma_u^{(j)})$ outside $E$ and set $g_j=0$ on $E$. Formula~\eqref{eq:avg_length}, applied separately to the finitely many strips $T_j=T_\delta(\gamma^{(j)})$, gives
\[
\frac1{2\delta}\int_{-\delta}^{\delta}\sum_{j=1}^k g_j(u)\,du
\le \frac{\sum_{j=1}^k A_{T_j}}{2\delta}\sqrt{\#S+1}
=:M_S.
\]
By the mean-value argument above, there is $u_0\notin E$ with $\sum_jg_j(u_0)\le M_S$. Since every summand is nonnegative, $g_j(u_0)\le M_S$ for every $j$. This is why we average the sum: selecting an averaging parameter separately for each loop would not provide a common base point.

Set $\widehat\gamma_j=\gamma_{u_0}^{(j)}$ and $q=q(u_0)$. Since the loops all avoid $S$ and pass through $q$, we also have $q\notin S$. The path
\[
\sigma(t)=\exp_{b_0}(tu_0\eta_0),\qquad 0\le t\le1,
\]
is defined also for $u_0=0$, when it is constant, and satisfies $\ell_\rho(\sigma)=|u_0|\le\delta$. For every $j$, the homotopy $\gamma_{tu_0}^{(j)}$ moves its base point along this same path. Consequently
\[
[\sigma^{-1}\circ\widehat\gamma_j\circ\sigma]
=[\gamma^{(j)}]=\alpha_j
\quad\text{in }\pi_1(B_0,b_0).
\]
Only the loops, not the conjugating path $\sigma$, are required to avoid $S$. All these homotopies take place in $B_0$.

Finally, take
\[
\mathcal K_0=\bigcup_{j=1}^kT_j,\qquad
L_0=\max_j\sup_{|u|\le\delta}\ell_\rho(\gamma_u^{(j)}),\qquad
\delta'=\delta.
\]
The finite union $\mathcal K_0$ is compactly contained in $B_0$, contains every selected loop and the normal segment $\sigma$, and $L_0$ is finite by smoothness and compactness. Taking $C=(\sum_jA_{T_j})/(2\delta)$ proves all the assertions, with constants independent of $S$.
\end{proof}

\subsection{Reduction to the simple case}\label{sec:reduction-to-simple}
In this subsection we lift the standing simplicity assumption of \S\ref{sec:growth} and prove \Cref{thm:new_bound_L} for every nontrivial class.

For $t\ge1$, $B_0$ has the homotopy type of a compact orientable surface of genus $g$ with $t$ boundary circles, so its fundamental group is free of rank $k=2g+t-1$ \cite[Theorem~3.2.8]{ZVC}. The disk-and-bands model gives a basis of embedded piecewise smooth loops. Smoothing each loop separately, with the base point fixed, gives smooth simple representatives of the same based classes. We do not require the smoothed representatives to remain pairwise disjoint away from the base point; this is unnecessary for the argument. For background on curves in surfaces see \cite[Chapter~1]{FM}. Up to this smoothing, this is the \emph{simple base} of \cite[\S3.1]{Phung}, which we call a \emph{simple basis}.

Our path convention is that $c_2\circ c_1$ first traverses $c_1$ and then $c_2$, whereas the product $[c_1][c_2]$ in the fundamental group is represented by $c_2\circ c_1$.

\paragraph{Proof of \Cref{thm:new_bound_L}, general case.}
First suppose $t\ge1$. Fix a nontrivial loop $\alpha$ in $B_0$.
Choose once and for all a representative of its free homotopy class
based at $b_0$, and denote its based class by $\alpha$ as well.
Fix a simple basis $\alpha_1,\dots,\alpha_k$ of $\pi_1(B_0,b_0)$
as above. Write $\alpha$ as a reduced word in this basis:
\[
\alpha = \alpha_{i_1}^{\epsilon_1}\cdots
\alpha_{i_m}^{\epsilon_m}\quad\text{in }\pi_1(B_0,b_0),
\qquad
\epsilon_j\in\{\pm1\},\quad i_j\in\{1,\dots,k\}.
\]
The word length $m\ge1$ is determined by the chosen based class
and the basis, which are fixed independently of $S$.

Fix $S\subset B_0$ with $\#S=s\ge0$. By
Proposition~\ref{prop:common-base} applied to the simple basis
$\alpha_1,\dots,\alpha_k$, there exist
\begin{itemize}
\item a point $q\in B_0\setminus S$,
\item a path $\sigma\colon[0,1]\to B_0$ from $b_0$ to $q$ with
$\ell_\rho(\sigma)\le\delta'$,
\item loops $\hat\gamma_1,\dots,\hat\gamma_k\subset B_0\setminus S$
based at $q$,
\end{itemize}
such that $[\sigma^{-1}\circ\hat\gamma_j\circ\sigma]=\alpha_j$
in $\pi_1(B_0,b_0)$ and
\[
\ell_S(\hat\gamma_j)\le C_0\sqrt{s+1}
\quad\text{for all }j=1,\dots,k,
\]
where $C_0>0$ depends only on $\alpha_1,\dots,\alpha_k$, $B_0$,
and $\rho$.

Define the concatenation
\[
\hat\gamma_\alpha:=
\hat\gamma_{i_m}^{\epsilon_m}\circ
\hat\gamma_{i_{m-1}}^{\epsilon_{m-1}}\circ\cdots\circ
\hat\gamma_{i_1}^{\epsilon_1},
\]
a piecewise smooth loop based at~$q$. We verify that
$\hat\gamma_\alpha\in\Gamma^\alpha_S$ and bound its Kobayashi length.

\smallskip
\noindent\textit{Membership in $\Gamma^\alpha_S$.}
Each $\hat\gamma_{i_j}$ is contained in $B_0\setminus S$, hence
so is the concatenation $\hat\gamma_\alpha$. To see that
$\hat\gamma_\alpha$ is freely homotopic to $\alpha$ in $B_0$,
insert the null-homotopic backtracking loop
$\sigma\circ\sigma^{-1}$ between consecutive factors.
The change-of-base-point map is a group homomorphism, so
\[
\begin{aligned}
[\sigma^{-1}\circ\hat\gamma_\alpha\circ\sigma]
&=[\sigma^{-1}\circ\hat\gamma_{i_1}^{\epsilon_1}\circ\sigma]
\cdots
[\sigma^{-1}\circ\hat\gamma_{i_m}^{\epsilon_m}\circ\sigma]\\
&=\alpha_{i_1}^{\epsilon_1}\cdots\alpha_{i_m}^{\epsilon_m}
=\alpha\quad\text{in }\pi_1(B_0,b_0).
\end{aligned}
\]
Thus $\hat\gamma_\alpha$, viewed as a loop based at $q$, is
conjugate (via the path $\sigma$) to a loop based at $b_0$
representing $\alpha$; equivalently, the free homotopy class
of $\hat\gamma_\alpha$ in $B_0$ coincides with the free homotopy
class of $\alpha$.

\smallskip
\noindent\textit{Kobayashi length bound.}
Since $\kappa_S$ is a Finsler pseudo-metric, its homogeneity gives
$\kappa_S(b,-v)=\kappa_S(b,v)$, hence
$\ell_S(\hat\gamma_{i_j}^{-1})=\ell_S(\hat\gamma_{i_j})$.
By additivity of $\ell_S$ over concatenation,
\[
\ell_S(\hat\gamma_\alpha)
=\sum_{j=1}^{m}\ell_S\bigl(\hat\gamma_{i_j}^{\epsilon_j}\bigr)
=\sum_{j=1}^{m}\ell_S(\hat\gamma_{i_j})
\le m\cdot C_0\sqrt{s+1}.
\]

\smallskip
\noindent\textit{Conclusion.}
Setting $C:=m\cdot C_0$, which depends only on the chosen based
class $\alpha$ (through its word length $m$ in the fixed basis),
the simple basis $\alpha_1,\dots,\alpha_k$, $B_0$, and $\rho$,
we obtain
\[
\ell_S(\Gamma^\alpha_S)
\le\ell_S(\hat\gamma_\alpha)\le C\sqrt{s+1}.
\]
As $S\subset B_0$ with $\#S=s\ge0$ was arbitrary, and since
sets with $\#S<s$ satisfy the same bound a fortiori
(the right-hand side being increasing in $s$), taking the
supremum over all $S$ with $\#S\le s$ yields
\[
L(\alpha,s)\le C\sqrt{s+1}\quad\text{for all }s\ge0.
\]
The case $S=\varnothing$ is already included in
Proposition~\ref{prop:common-base}.

If $t=0$, choose a fixed small closed disk
$\overline\Delta\subset B$ disjoint from a piecewise smooth
representative of $\alpha$, independently of $S$, and put
$U=B\setminus\overline\Delta$. The class remains nontrivial
in $U$, since a null-homotopy there would also be one in $B$.
Apply the bordered result to $U$ and $S\cap U$. The resulting
loop avoids all of $S$, and the inclusion
$U\setminus(S\cap U)\hookrightarrow B\setminus S$ decreases
Kobayashi lengths. Its free homotopy class in $B$ is that
of $\alpha$. This proves the assertion for compact hyperbolic
$B$ as well.\qed
\medskip

\begin{cor}\label{cor:asympt_impr}
For every nontrivial $\alpha$ as above,
\[
\lim_{s\to\infty}\frac{\log L(\alpha,s)}{\log s}=\frac12.
\]
\end{cor}
\begin{proof}
Combine Theorem~\ref{thm:new_bound_L} with
\cite[Theorem~C]{Phung}:
\[
\frac{c\sqrt{s}}{\log(s+2)}
\le L(\alpha,s)\le C\sqrt{s+1}.
\]
Taking logarithms and dividing by $\log s$ proves the assertion.
\end{proof}

\begin{rem}\label{rem:sharp_exponent}
The quoted lower bound and our upper bound differ by a factor
of order $\log(s+2)$. These estimates determine the polynomial
growth exponent, but do not determine a leading constant or
imply $L(\alpha,s)\asymp\sqrt{s}$.
\end{rem}

\section{Bounding generalized integral points}\label{sec:phung_proof}
\subsection{Phung's bound}
In this section we revisit the proof of Theorem~\ref{thm:main_phung}. We follow Phung's strategy \cite{Phung} closely, and we recall the structure of the argument in enough detail to identify where our improvement enters (\S\ref{sec:improved}). Our account differs from \emph{loc.\ cit.}\ in the following respects; each point is taken up again, with full notation, at the corresponding place of the proof below.
\begin{itemize}
\item Phung bounds the Kobayashi distance of $(\mathcal A\smallsetminus\mathcal D)|_{B_0}$ from below by a fixed hermitian metric by invoking \cite[Theorem~7.5]{Phung}, a criterion of Brody--Green type. We give a short self-contained proof of the infinitesimal comparison actually needed here, based on Brody's reparametrization lemma and Hurwitz's theorem \cite[Lemmas~7.2.1 and~7.2.11]{NW}. In particular, we check the hyperbolicity of the two possible strata containing a limiting entire curve.
\item At one point the argument requires a path of uniformly bounded length joining two points of the same fibre \emph{within that fibre}. In \emph{loc.\ cit.}\ this length is bounded by the diameter of the closure of $\mathcal A_{B_0}$ in $\mathcal A$. This is not sufficient: two points of a fibre may be close to each other in $\mathcal A$ even though every path joining them inside the fibre is long. We bound the length instead by the diameter of the fibre itself, measured with respect to the metric induced on the fibre, and we check that these diameters are uniformly bounded.
\item Phung converts bounds on the lengths of certain loops in $\mathcal A_{B_0}$ into bounds on the displacements, along the fibre direction of the universal cover, of the corresponding deck transformations. This conversion needs justification: the metric pulled back to the universal cover need not be a product metric, so the projection onto the fibre direction does not decrease lengths in general. We show that all the lifted paths in question stay in a fixed region, independent of $s$ and of the section, on which the differential of this projection has a uniform bound, giving the required estimate for lengths of projected paths.
\item Several routine reductions are carried out in detail: the case where $\pi_1(B_0)$ is trivial, the reduction to $t\ge1$, the enlargement of the disks $\overline\Delta_i$ absorbing the (countable) set of points of $B$ whose fibre in $\mathcal D$ fails to be hyperbolic, and the verification that the two conjugating paths naturally appearing in the construction differ by a loop whose homotopy class lies in the fibre subgroup, so that the associated cocycles are cohomologous.
\end{itemize}
The statement and the strategy of Theorem~\ref{thm:main_phung} are entirely Phung's; the purpose of the details above is to make the argument self-contained in the form needed for \S\ref{sec:improved}.

\paragraph{Proof of Theorem \ref{thm:main_phung}.} Fix a hermitian metric $h$ on the smooth complex manifold $\mathcal A_U$, where $U:=B\smallsetminus\Sing(f)$. Let $k$ be the minimal number of generators of $\pi_1(B_0,b_0)$, as in the introduction.

\noindent\textit{The rank-zero case.} If $k=0$ then $\pi_1(B_0,b_0)$ is trivial, so $H^1(\pi_1(B_0,b_0),\Gamma)=0$ for the fibre lattice $\Gamma=H_1(A_{b_0},\mathbb Z)\cong\mathbb Z^{2n}$. By the Parshin--Phung injectivity (\cite[Proposition~1]{Parshin}, \cite[Proposition~7.2]{Phung}; cf.\ Proposition~\ref{prop:parshin-phung} below) the map $A(K)/\Tr_{K/\mathbb C}(A)(\mathbb C)\to H^1(\pi_1(B_0,b_0),\Gamma)=0$ is the zero homomorphism; hence every two points of $A(K)$ are congruent modulo $\Tr_{K/\mathbb C}(A)(\mathbb C)+A(K)_{\mathrm{tors}}$, so $A(K)/\Tr_{K/\mathbb C}(A)(\mathbb C)$ is finite of order $t_\mathcal A:=\#\bigl(A(K)/\Tr_{K/\mathbb C}(A)(\mathbb C)\bigr)_{\mathrm{tors}}<\infty$ (Lang--N\'eron). Consequently $\#\bigl(I(s,B_0)/\Tr_{K/\mathbb C}(A)(\mathbb C)\bigr)\le t_\mathcal A$ for every $s$, and the bounds of Theorem~\ref{thm:main_phung} and Theorem~\ref{thm:main} hold trivially (take $m\ge t_\mathcal A$). We therefore assume $k\ge 1$ from now on.

\noindent\textit{Reduction to $t\ge 1$.} As in the proof of \cite[Theorem~A]{Phung}, we may moreover assume $t\ge 1$: if $t=0$ (so that $\Sing(f)=\varnothing$ and $B_0=B$ is a compact hyperbolic surface of genus $g\ge 2$), remove one closed disk $\overline\Delta_1\subset B$. This shrinks $B_0$, so $I(s,B_0)$ only grows, and it leaves the rank unchanged, $\rk\pi_1(B\smallsetminus\overline\Delta_1)=2g+1-1=2g=\rk\pi_1(B)$; hence the bound for the new $B_0$ implies the bound for the old one. With $t\ge 1$ the surface $B_0$ has the homotopy type of a compact surface of genus $g$ with $t\ge 1$ boundary circles, so $\pi_1(B_0,b_0)$ is a \emph{free} group of rank $k=2g+t-1$; this freeness will be used in Steps~3 and~6, and it places us within the standing hypotheses of \S\ref{sec:growth}.

\noindent\textit{Reduction to $s\ge 1$.} Since $I(0,B_0)\subseteq I(1,B_0)$, the bound for $s=1$ implies the bound for $s=0$ after enlarging $m$. We therefore assume $s\ge 1$ from now on.\\

 \noindent\textit{Step 1: Reduction to the hyperbolic locus.} Consider the non-hyperbolic locus
\[
V:=\{b\in B\colon (\mathcal D_b)_{\mathrm{red}} \text{ is not hyperbolic}\}
\]
It is well known that $V$ is a closed subset of $B$ in the analytic (Euclidean) topology: the hyperbolic locus of a proper holomorphic map between complex spaces is open (\cite[Theorem~7.2.15]{NW}, applied to $f|_{\mathcal D}\colon\mathcal D\to B$). Moreover, by \cite[Lemma~B.2]{Phung} applied to the abelian scheme $\mathcal A|_{B\smallsetminus\Sing(f)}$, via the inclusion $V\smallsetminus\Sing(f)\subseteq Z\bigl(\mathcal A|_{B\smallsetminus\Sing(f)},\mathcal D|_{B\smallsetminus\Sing(f)}\bigr)$ furnished by \cite[Theorem~7.6]{Phung}, and using crucially the assumption \texttt{(P)}(iii) that $D_{\overline K}$ contains no translate of a positive-dimensional abelian subvariety, the set $V\smallsetminus\Sing(f)$ --- and hence $V$, as $\Sing(f)$ is finite --- is \emph{at most countable}. The geometric form of (P)(iii) is used here: a dominating component of the scheme parametrizing positive-dimensional translates would produce such a translate after a finite extension of $K$.

We may therefore enlarge the disks $\overline\Delta_1,\dots,\overline\Delta_t$ so as to remove $V\cup\Sing(f)$ from the base. First, a collar of each $\partial\overline\Delta_i$ is foliated by parallel Jordan curves, and only countably many of these meet the countable set $V\cup\Sing(f)$; we may thus enlarge each $\overline\Delta_i$ inside its collar to a closed disk $\overline\Delta_i^+$ whose boundary avoids $V\cup\Sing(f)$, the $\overline\Delta_i^+$ remaining pairwise disjoint. In particular $\Sing(f)\subset\bigsqcup_i\operatorname{int}\overline\Delta_i^+$. The set $W:=(V\cup\Sing(f))\smallsetminus\bigsqcup_i\operatorname{int}\overline\Delta_i^+$ is closed and disjoint from the compact set $\bigsqcup_i\overline\Delta_i^+$. If $W$ is empty, set $\overline\Delta_i'=\overline\Delta_i^+$ and omit the following attachments. Otherwise its $\rho$-distance $d_0$ from that compact set is positive. Applying \cite[Lemma~8.1]{Phung} to $W$ with disks of radius at most $\varepsilon$, for some $0<\varepsilon<d_0/2$, yields finitely many pairwise disjoint closed cover disks whose union contains $W$ in its interior and each of which is disjoint from every $\overline\Delta_j^+$. The boundaries can be chosen disjoint from $W$, as in the proof of that lemma, so that the interiors cover $W$. Annex the cover disks successively to the first enlarged original disk: join the current disk to the next cover disk $V'$ by an embedded arc avoiding the other disks, and take a small regular neighborhood of their union. Such an arc exists because the complement of finitely many disjoint disks is connected. Equivalently, attach $V'$ by a thin tube $\beta$ meeting the two disks only in short boundary arcs at its ends. At each stage the resulting region is again a closed disk. Denote the final enlarged disks by $\overline\Delta_i'$. Each $\overline\Delta_i'$ is simply connected with piecewise-$C^1$ boundary, hence again a closed disk (in general not a metric ball), and the $\overline\Delta_i'$ are pairwise disjoint, with distinct points of $\Sing(f)$ still lying in distinct disks. By construction $V\cup\Sing(f)\subseteq\bigsqcup_{i=1}^t\operatorname{int}\overline\Delta_i'$. Writing $B_0':=B\smallsetminus\bigsqcup_{i=1}^t\overline\Delta_i'\subseteq B_0$, the inclusion $B_0'\hookrightarrow B_0$ is a homotopy equivalence: the region between each original disk and the corresponding enlarged disk is an annulus, which retracts to its outer boundary on the compact surfaces with boundary; collars identify each such compact surface up to homotopy with its interior. In particular the induced map on fundamental groups is an isomorphism, so $\rk\pi_1(B_0')=\rk\pi_1(B_0)=k$ is unchanged. Since $B_0'\subseteq B_0$ we have $I(s,B_0)\subseteq I(s,B_0')$ for every $s$, so it suffices to prove the bound for $B_0'$. Replacing $B_0$ by $B_0'$ and choosing $b_0$ in the new surface, we may thus assume from now on that $V\cup\Sing(f)\subseteq\bigsqcup_{i=1}^t\operatorname{int}\overline \Delta_i$; equivalently, the compact set $\overline{B_0}$ is disjoint from $V\cup\Sing(f)$. In particular $(\mathcal D_{b})_{\mathrm{red}}$ is Kobayashi hyperbolic for every $b\in \overline{B_0}$; hence $\mathcal D_b$ contains no translate of a positive-dimensional subtorus of $\mathcal A_b$, since such a translate would be a non-hyperbolic complex subspace of the hyperbolic space $\mathcal D_b$. By a result of Green (\cite[Theorem~7.3.8(ii)]{NW}, applied to the underlying reduced divisor of $\mathcal D_b$; see also \cite[Theorem~7.6]{Phung}) it follows that $\mathcal A_b\setminus\mathcal D_b$ is hyperbolically embedded into $\mathcal A_b$, in particular Kobayashi hyperbolic, for any $b\in \overline{B_0}$. Finally, fix a compact connected smooth subsurface $N\Subset B_0$ with $b_0\in N$ (it will be enlarged below finitely many times, by fixed compact sets independent of $s$ and $P$). The conjugating paths $c_{b_0,b}$ are chosen as in the finite reference-loop construction of \cite[Lemma~4.1]{Phung}: choose a fixed path from $b_0$ to each reference base point, then move to $b$ within the corresponding simply connected neighborhood. These choices give a collection $\mathcal C_{b_0}$ of paths in a fixed compact set, with uniformly bounded $\rho$-length; the same path is used for all generators. Enlarge $N$ to contain them and choose $\delta_N>0$ so that
\[
\ell_\rho(c_{b_0,b})\le\delta_N<+\infty\,.
\]
\\

\noindent\textit{Step 2: Bounding loop lengths in the base.}
Let $P\in I(s, B_0)$ and set $S:=f(\sigma_P(B_0)\cap\mathcal D)\cap B_0$, so that $\#S\le s$ and $\sigma_P(B_0\smallsetminus S)\subseteq(\mathcal A\smallsetminus\mathcal D)|_{B_0\smallsetminus S}$; we may assume $S\neq\emptyset$ (otherwise replace $S$ by an arbitrary singleton of $B_0$: this only increases the Kobayashi lengths below, and $\#S\le s$ still holds since $s\ge 1$). Let $\alpha_1,\ldots,\alpha_k$ be a simple basis of $\pi_1(B_0,b_0)$. By \cite[Theorem 5.1]{Phung}, applied with $\mathcal C_{b_0}$ as the fixed collection of paths of \cite[Definition~3.2]{Phung} (with the paths chosen in Step~1; the base points $b$ produced there lie in $B_\varepsilon\subseteq N$, cf.\ below), there exist $b\in B_0$ and piecewise smooth loops $\gamma_j\in \Gamma^{\alpha_j}_S$ such that $\alpha_j=c_{b_0,b}^{-1}\,\gamma_j\,c_{b_0,b}$ in $\pi_1(B_0,b_0)$ where $c_{b_0,b}\in\mathcal C_{b_0}$ and the following bound holds:
\begin{equation}\label{eq:phung_lin_bound}
\ell_S(\gamma_j)\le L(s+1)\,\quad\forall j=1,\ldots, k
\end{equation}
for $L\in\mathbb R_{>0}$ independent of $s,S,b$. Two further facts, immediate from the proof of \cite[Theorem~5.1]{Phung}, will be needed in Step~5: the loops $\gamma_j$ lie in a fixed compact set $\mathcal K_1\Subset B_0$ and satisfy $\ell_\rho(\gamma_j)\le L_1$, with $\mathcal K_1$ and $L_1$ independent of $s,S,P$. Indeed, by \cite[Lemmas~3.7 and~4.1]{Phung} the $\gamma_j$ are obtained by modifying, within $\rho$-distance $2a$, finitely many fixed loops of bounded $\rho$-length, while keeping $\rho$-distance $\ge 2a$ from $\partial B_0$; the length bound follows from \cite[(5.2)--(5.3)]{Phung}. Moreover, the base points $b$ lie in the fixed compact set $B_\varepsilon\Subset B_0$ of \emph{loc.\ cit.} We enlarge $N$ once and for all so that $\mathcal K_1\cup B_\varepsilon\subseteq N$.

There exists $c\in\mathbb R_{>0}$ such that, at the infinitesimal level,
\begin{equation}\label{eq:metric_comparison}
\kappa_{(\mathcal A\smallsetminus\mathcal D)|_{B_0}}(x,v)\ge c\, |v|_h\qquad\forall\,(x,v)\,.
\end{equation}
This is the infinitesimal comparison corresponding to \cite[Theorem~7.5]{Phung} (stated there at the level of distances), a criterion of Brody--Green type; see also \cite{Green}. We give the short argument directly for our family, making explicit why a limiting entire curve must lie in one of the two hyperbolic strata of a fibre. If no such $c$ exists then, by the definition of the Kobayashi--Royden metric (\S\ref{sec:kob_metric}), there are holomorphic disks $g_\nu\colon\Delta(0,1)\to(\mathcal A\smallsetminus\mathcal D)|_{B_0}$ with $|g_\nu'(0)|_h\to\infty$. Since $(\mathcal A\smallsetminus\mathcal D)|_{B_0}$ is relatively compact in the complex manifold $\mathcal A|_{B\smallsetminus\Sing(f)}$ (its closure lies over $\overline{B_0}$, which is disjoint from $\Sing(f)$ by the reduction above), Brody's reparametrization lemma \cite[Lemma~7.2.1]{NW} produces a non-constant entire curve $F\colon\mathbb C\to\mathcal A_U$ which is a locally uniform limit of holomorphic maps with values in $(\mathcal A\smallsetminus\mathcal D)|_{B_0}$. In particular $f\circ F(\mathbb C)\subseteq\overline{B_0}$ omits the nonempty open set $\operatorname{int}\overline\Delta_1$, hence omits three points of $B$; so $f\circ F$ maps into the complement of three points of $B$, which is hyperbolic by uniformization, and $f\circ F$ is therefore constant. Hence $F(\mathbb C)\subseteq\mathcal A_b$ for some $b\in\overline{B_0}$, and $b\notin V\cup\Sing(f)$ by the reduction above. By Hurwitz's theorem in the form \cite[Lemma~7.2.11]{NW}, applied on an exhaustion of $\mathbb C$ by disks inside the complex manifold $\mathcal A|_{B\smallsetminus\Sing(f)}$, where $\mathcal D$ restricts to a reduced divisor, either $F(\mathbb C)\subseteq\mathcal D_b$ or $F(\mathbb C)\subseteq\mathcal A_b\smallsetminus\mathcal D_b$; both are Kobayashi hyperbolic since $b\notin V$ (Step~1), so $F$ is constant --- a contradiction.

The divisor is locally Cartier on the smooth total space $\mathcal A_U$, so this use of Hurwitz is legitimate even if the compact model is singular.

Now notice that 
\[
\sigma_P(\gamma_j)\subset \sigma_P(B_0\smallsetminus S)\subseteq(\mathcal A\smallsetminus\mathcal D)|_{B_0\smallsetminus S}
\]
So by \eqref{eq:phung_lin_bound}, \eqref{eq:metric_comparison}, and the infinitesimal contraction property of the Kobayashi--Royden metric, integrated along $\gamma_j$, we deduce
\begin{equation}\label{eq:phung_looplength}
\ell_{h}(\sigma_P(\gamma_j))\le c^{-1}L(s+1)\,.
\end{equation}

\medskip\noindent\textit{Step 3: The Parshin cocycle.} Since $f\colon\mathcal A_{B_0}\to B_0$ is a proper smooth submersion, Ehresmann's theorem gives a fiber bundle in the differentiable category. Let $\sigma_O$ be the zero section and fix $w_0:=\sigma_O(b_0)\in \mathcal A_{b_0}$. Since $B_0$ is a $K(\pi,1)$-space (i.e.\ $\pi_i(B_0)=0$ for $i\ge 2$, which holds because its universal cover is the disc $\mathbb D$), the long exact sequence of homotopy groups for the fibration $\mathcal A_{b_0}\to\mathcal A_{B_0}\to B_0$ yields the short exact sequence
\begin{equation}\label{eq:ses}
1\longrightarrow \pi_1(\mathcal A_{b_0},w_0)\longrightarrow\pi_1(\mathcal A_{B_0},w_0)\xrightarrow{\;f_*\;}\pi_1(B_0,b_0)\longrightarrow 1\,.
\end{equation}
As $\mathcal A_{b_0}$ is a complex torus of dimension $n$, we have $\pi_1(\mathcal A_{b_0},w_0)=H_1(\mathcal A_{b_0},\mathbb Z)=:\Gamma\cong\mathbb Z^{2n}$. The zero section $\sigma_O$ splits~\eqref{eq:ses} via $i_O(\alpha):=[\sigma_O(\gamma)]$ where $\alpha=[\gamma]\in G:=\pi_1(B_0,b_0)$. Each rational point $P\in A(K)$ with section $\sigma_P\colon B_0\to\mathcal A_{B_0}$ gives a second splitting
\[
i_P(\alpha):=[\eta_{w_0,\sigma_P(b_0)}^{-1}\circ\sigma_P(\gamma)\circ\eta_{w_0,\sigma_P(b_0)}]\,,
\]
where $\eta_{w_0,\sigma_P(b_0)}$ is a path in the fiber $\mathcal A_{b_0}$ from $w_0$ to $\sigma_P(b_0)$. Since both $i_P$ and $i_O$ are sections of $f_*$, for every $\alpha\in G$:
\[
f_*\bigl(i_P(\alpha)\cdot i_O(\alpha)^{-1}\bigr)=\alpha\cdot\alpha^{-1}=1\,,
\]
so $i_P(\alpha)\cdot i_O(\alpha)^{-1}\in\ker(f_*)=\Gamma$. We define the \emph{Parshin cocycle}
\[
c_P\colon G\to\Gamma\,,\quad c_P(\alpha):=i_P(\alpha)\cdot i_O(\alpha)^{-1}\,.
\]
One verifies that $c_P$ is a $1$-cocycle of $G$ with coefficients in the $G$-module $\Gamma$, where the $G$-action on $\Gamma$ is  by conjugation
\[
\begin{aligned}
\varphi\colon G&\to \Aut(\Gamma)\\
\alpha&\mapsto \varphi_\alpha\colon\gamma \mapsto i_O(\alpha)\,\gamma\,i_O(\alpha)^{-1}
\end{aligned}
\]
Note that the action $\varphi$ doesn't depend on the choice of the splitting $i_O$, in fact for another splitting $i_P$, we have $i_P(\alpha)=c_P(\alpha)\cdot i_O(\alpha)$ with $c_P(\alpha)\in\Gamma$, and
\[
i_P(\alpha)\,\gamma\,i_P(\alpha)^{-1}
=c_P(\alpha)\cdot\varphi_\alpha(\gamma)\cdot c_P(\alpha)^{-1}
=\varphi_\alpha(\gamma)\,,
\]
where the last equality uses the commutativity of $\Gamma$. The cohomology class  $[c_P]\in H^1(G,\Gamma)$ is independent of the choice of path $\eta_{w_0,\sigma_P(b_0)}$ (a different choice changes $c_P$ by a coboundary). The following result, which goes back to Parshin \cite{Parshin} and is proved in \cite[Proposition~7.2]{Phung}, will be crucial:

\begin{prop}[Parshin-Phung] \label{prop:parshin-phung}
  The homomorphism
  \[
  \begin{aligned}
    \Psi\colon A(K)&\to H^1(G,\Gamma)\\
    P&\mapsto [c_P]
    \end{aligned}
    \]
factors through $A(K)/\mathrm{Tr}_{K/\mathbb{C}}(A)(\mathbb{C})$ and has the following injectivity property: if $\Psi(P) = \Psi(Q)$, then $P - Q \in \mathrm{Tr}_{K/\mathbb{C}}(A)(\mathbb{C}) + A(K)_{\mathrm{tors}}$. In particular, each element $[c] \in H^1(G, \Gamma)$ in the image of~$\Psi$ arises from at most
$t_\mathcal{A} := \#\bigl(A(K)/\mathrm{Tr}_{K/\mathbb{C}}(A)(\mathbb{C}) \bigr)_{\mathrm{tors}} < \infty$ rational points modulo the trace (finite by the Lang--N\'eron theorem).
\end{prop}
\begin{rem}\label{rem:parshin-coefficients}
The homomorphism $\Psi$ is the connecting map of the relative exponential sequence \cite[Equation~(7.3)]{Phung}, so it is additive. Since $G$ is free of finite rank, $H^1(G,\Gamma)$ is the cokernel of $\Gamma\to\Gamma^k$, $\lambda\mapsto(\varphi_{\alpha_j}(\lambda)-\lambda)_j$, and is finitely generated. The divisible group $T_A$ therefore maps to zero. This cokernel need not be torsion-free, contrary to the assertion used in the proof of \cite[Proposition~7.2]{Phung}.

For completeness, that assertion can be avoided as follows. If $\Psi(P)=0$, then for every $m\ge2$ its reduction modulo $m$ is the zero Kummer class of $[m]^{-1}(\sigma_P)$ over $B_0$, using $\Gamma/m\Gamma\simeq\mathcal A[m]_{b_0}$ as monodromy modules; this follows by dividing local logarithms by $m$, as in \cite[Proposition~7.1]{Phung}. Thus this finite \'etale torsor has an analytic section over $B_0$. The map $\pi_1(B_0)\to\pi_1(U)$ is surjective by van Kampen, since each removed disk contains at most one point of $B\setminus U$. Hence its monodromy fixes a sheet over $U$, giving a degree-one component and an algebraic section over $U$. Consequently $P\in mA(K)$ for every $m$, and its image in the free abelian group $M/M_{\mathrm{tors}}$ is zero. This gives the stated finite-kernel conclusion without a torsion-freeness assumption on $H^1(G,\Gamma)$. Here torsion classes in $M$ lift to torsion points modulo $T_A$, because $T_A$ is divisible.
\end{rem}

\begin{rem}
It would be interesting to understand the relationship between the Parshin cocycle and the ``analytic cocycle'' induced by the logarithm of the section $\sigma_P$, as defined in \cite[Equation (13)]{DT}.
\end{rem}
To extract a quantitative bound from the cocycle $c_P$, we work in the semidirect product coordinates provided by $i_O$. The splitting $i_O$ exhibits
\begin{equation}\label{eq:semidirect}
\pi_1(\mathcal A_{B_0},w_0)\;\cong\;\Gamma\rtimes_\varphi G\,.
\end{equation}
In these coordinates, $i_O(\alpha_j)=(0,\alpha_j)$ and $i_P(\alpha_j)=(\beta_j,\alpha_j)$ for some $\beta_j\in\Gamma$, so the Parshin cocycle reads $c_P(\alpha_j)=\beta_j$. Since $G$ is free, the $k$-tuple $(\beta_1,\ldots,\beta_k)\in\Gamma^k$ determines $c_P$ completely. After the simultaneous change of splitting in Step~4, the displacement bound in Step~5 will give a norm bound on each $\beta_j$.

\medskip\noindent\textit{Step 4: Bounding loop lengths in $\mathcal A$.} Let us now go back to the special loops (based in $b$) $\gamma_j\in\Gamma_S^{\alpha_j}$  for $j=1,\ldots, k$  considered in Step 2. Consider the following loop based at $w_0$:
\[
\sigma_P(\gamma_j)^\#:=( v_b\circ\sigma_O(c_{b_0,b}))^{-1}\circ \sigma_P(\gamma_j)\circ( v_b\circ\sigma_O(c_{b_0,b}))
\]
where $v_b\subset\mathcal A_b$ is an intrinsic $h$-geodesic from $\sigma_O(b)$ to $\sigma_P(b)$. Since $\sigma_P(\gamma_j)$ is homotopic relative to its base point to $\sigma_{P}(c_{b_0,b})\circ\sigma_P(\alpha_j)\circ\sigma_{P}(c_{b_0,b})^{-1}$ (with a loop representative of $\alpha_j$ understood), $\sigma_P(\gamma_j)^\#$ is freely homotopic to $\sigma_{P}(\alpha_j)$ and each of its components satisfies the following  uniform bounds:
\begin{itemize}
\item the loop $\sigma_P(\gamma_j)$ satisfies:
\[
\ell_h(\sigma_P(\gamma_j))\le c^{-1} L(s+1)
\]
by \Cref{eq:phung_looplength};
\item the fiber path $v_b\subset\mathcal A_{b}$ satisfies:
\[
\ell_h(v_b)\le\delta_0:=\sup_{b\in\overline{B_0}}\diam_h(\mathcal A_b)<+\infty\;;
\]
(here $\diam_h(\mathcal A_b)$ is the \emph{intrinsic} $h$-diameter of the fibre, which bounds the length of a fibre geodesic; the supremum is finite because $\overline{B_0}\cap\Sing(f)=\varnothing$ by Step~1. On finitely many Ehresmann trivializations covering this compact base, the vertical metric is uniformly comparable with a fixed metric on the reference compact fibre, so its intrinsic diameter is uniformly bounded)
\item since $N\Subset U$ is compact $\Vert d\sigma_O \Vert_{\infty}:=\sup_{b\in N}\Vert d_b\sigma_O\Vert_{\rho,h}<+\infty$, so
\[
\ell_h(\sigma_O(c_{b_0,b}))\le \Vert d\sigma_O \Vert_{\infty}\,\ell_\rho(c_{b_0,b})\le\Vert d\sigma_O \Vert_{\infty}\,\delta_N=:\delta_0'
\]
\end{itemize}
In other words for any element of the simple basis $\alpha_j$ we can construct a loop $\sigma_P(\gamma_j)^\#$ based at $w_0$ that is freely homotopic to $\sigma_P(\alpha_j)$ and satisfies
\[
\ell_h\left(\sigma_P(\gamma_j)^\#\right)\le H(s):=c^{-1}L(s+1)+2(\delta_0+\delta_0')
\]
where the constants $L,\delta_0,\delta_0'$ don't depend on $S$ and $P$.\\

\medskip
\noindent
The loop $\sigma_P(\gamma_j)^\#$ is based at~$w_0$, so it defines an element
$[\sigma_P(\gamma_j)^\#] \in \pi_1(\mathcal{A}_{B_0}, w_0)$.
This element need not equal $i_P(\alpha_j) = (\beta_j, \alpha_j)$, because
$\sigma_P(\gamma_j)^\#$ and the loop defining~$i_P(\alpha_j)$ use different
conjugating paths from~$w_0$ to the fibre above~$b$: the former uses
$v_b \circ \sigma_O(c_{b_0,b})$, while the latter uses
$\sigma_P(c_{b_0,b}) \circ \eta_{w_0,\sigma_P(b_0)}$.  Both are paths from~$w_0$ to~$\sigma_P(b)$, so
their concatenation
\[
  \omega
  \;:=\;
  \bigl(v_b \circ \sigma_O(c_{b_0,b})\bigr)^{-1}
  \circ\;
  \sigma_P(c_{b_0,b})
  \circ\;
  \eta_{w_0,\sigma_P(b_0)}
\]
is a loop at~$w_0$ with $[\omega]\in\Gamma$: indeed $f_*[\omega]=1$, since $f\circ v_b$ and $f\circ\eta_{w_0,\sigma_P(b_0)}$ are constant while $f\circ\sigma_O(c_{b_0,b})=f\circ\sigma_P(c_{b_0,b})=c_{b_0,b}$. A direct verification gives
\[
  [\sigma_P(\gamma_j)^\#]
  \;=\;
  [\omega]^{-1} \cdot i_P(\alpha_j) \cdot [\omega]
  \qquad\text{in } \pi_1(\mathcal{A}_{B_0}, w_0)\,.
\]
The same $\omega$ is used for all $j$. Its homotopy class lies in the fibre subgroup, although the displayed loop need not lie inside a fibre. In other words, $[\sigma_P(\gamma_j)^\#]$ and $i_P(\alpha_j)$ are
\emph{conjugate} in $\pi_1(\mathcal{A}_{B_0}, w_0)$.  The map
$i_P' \colon G \to \pi_1(\mathcal{A}_{B_0}, w_0)$ defined by
$i_P'(\alpha_j) := [\sigma_P(\gamma_j)^\#]$ is therefore a section of~\eqref{eq:ses} (conjugation by $[\omega]\in\Gamma=\ker f_*$ preserves the section property), $\Gamma$-conjugate to~$i_P$, so the associated cocycles are
cohomologous: $[c'_P] = [c_P] \in H^1(G, \Gamma)$.  By
Proposition~\ref{prop:parshin-phung}, we may freely replace~$i_P$
by~$i_P'$ for counting purposes.  We do so and, to lighten notation, simply
write~$\beta_j$ for the components of the new cocycle.

\medskip
\noindent
\textit{Step 5: Displacement bound via the universal cover.}
We now explain how the $h$-length bound on $\sigma_P(\gamma_j)^\#$ constrains
the lattice element $\beta_j \in \Gamma$. Let $\pi\colon \widetilde{\mathcal A}_0\to\mathcal A_{B_0}$ be the universal cover. Since $\mathcal{A}_{B_0} \to B_0$ is a fibre bundle (by Ehresmann's theorem)
with fibre $\mathcal A_{b_0}$ (a complex torus of dimension~$n$, so
$\widetilde{\mathcal A_{b_0}} \cong \mathbb{R}^{2n}$) and base~$B_0$ (a hyperbolic
Riemann surface, so $\widetilde{B_0} \cong \mathbb D$, the unit disk), pulling
back to~$\mathbb D$ trivialises the bundle and gives
$\widetilde{\mathcal{A}}_0 \cong \mathbb{R}^{2n} \times \mathbb D$. We use the
trivialisation induced by the relative exponential (cf.\ \cite[(7.3)]{Phung}):
the lattice local system $(R^1f_*\mathbb Z)^\vee$ is trivial over the
contractible~$\mathbb D$, and with this choice the subgroup
$\Gamma\subset\pi_1(\mathcal A_{B_0},w_0)$ acts on
$\mathbb R^{2n}\times\mathbb D$ by $(\tilde x,\tilde y)\mapsto(\tilde x+\gamma,\tilde y)$.
 
Fix the point $\tilde{w} = (\tilde{x}_0, \tilde{y}_0) \in \pi^{-1}(w_0)$
with $\tilde x_0=0$: this choice is possible because
$w_0=\sigma_O(b_0)$ lies on the zero section, whose lift in the chosen
trivialisation is $\{0\}\times\mathbb D$, and $\tilde y_0$ is any point
of $\mathbb D$ lying over $b_0$ for the universal covering
$v\colon\mathbb D\to B_0$.
Pull back the hermitian metric~$h$ to a Riemannian metric~$\tilde{h}$
on~$\widetilde{\mathcal{A}}_0$.
 The group of deck transformations of~$\pi$ is
$\pi_1(\mathcal{A}_{B_0}, w_0) \cong \Gamma \rtimes_\varphi G$ (with the
composition convention for $\pi_1$ already implicit in our loop
concatenations). The action of an element
$(\beta, \alpha)=\beta\cdot i_O(\alpha) \in \Gamma \rtimes_\varphi G$ on
$\mathbb R^{2n}\times\mathbb D$ is given by (homotopy lifting property,
cf.\ the proof of \cite[Theorem~A]{Phung}):
\begin{equation}\label{eq:deck-action-general}
  (\beta, \alpha) \cdot (\tilde{x}, \tilde{y})
  \;=\;
  \bigl(\varphi_\alpha(\tilde x)+\beta,\;\;
         \alpha \cdot \tilde{y}\bigr)\,,
\end{equation}
where $\alpha \cdot \tilde{y}$ is the deck transformation of
$v \colon \mathbb D \to B_0$ by $\alpha \in G = \pi_1(B_0, b_0)$ (acting on
the base), and $\varphi_\alpha$ acts on the fibre factor
$\mathbb R^{2n}$ as the linear extension of the monodromy: the fibre
component ``twists'' as one transports along the base loop~$\alpha$, and
the semidirect product structure of $\Gamma \rtimes_\varphi G$ encodes
precisely this twisting. In particular the fibre component of the action
is a translation only when $\varphi_\alpha=\mathrm{id}$; at the chosen
point $\tilde w=(\tilde x_0,\tilde y_0)$, however, thanks to the
normalisation $\tilde x_0=0$, \eqref{eq:deck-action-general} reads
\begin{equation}\label{eq:deck-action}
  (\beta, \alpha) \cdot (\tilde{x}_0, \tilde{y}_0)
  \;=\;
  \bigl(\beta \cdot \tilde{x}_0,\;\;
         \alpha \cdot \tilde{y}_0\bigr)\,,
\end{equation}
where $\beta\cdot\tilde x_0$ is the deck transformation of
$u \colon \mathbb{R}^{2n} \to \mathcal A_{b_0}$ by the element
$\beta \in \Gamma$ (acting on the fibre).

The loop $\sigma_P(\gamma_j)^\#$ is based at~$w_0$ and has homotopy class
$(\beta_j, \alpha_j)$.  Its lift to~$\widetilde{\mathcal{A}}_0$
from~$\tilde{w}$ is a path from~$\tilde{w}$ to
$(\beta_j, \alpha_j) \cdot \tilde{w}$.  Since~$\pi$ is a local isometry for
the metrics $\tilde{h}$ and~$h$, the $\tilde{h}$-length of the lift equals
$\ell_h(\sigma_P(\gamma_j)^\#) \leq H(s)$, so:
\begin{equation}\label{eq:displacement}
  d_{\tilde{h}}\!\bigl(\tilde{w},\;
    (\beta_j, \alpha_j) \cdot \tilde{w}\bigr)
  \;\leq\;
  H(s) \;:=\; c^{-1}L(s+1) + 2(\delta_0 + \delta_0')
  \;=\; O(s)\,,
\end{equation}
where $L$, $\delta_0$, $\delta_0'$ are independent of~$S$ and~$P$. Let
$d_j$ be the $\Gamma$-invariant geodesic metric on the fibre factor
$\mathbb R^{2n}$ obtained by restricting $\tilde h$ to the slice
$\mathbb R^{2n}\times\{\alpha_j\cdot\tilde y_0\}$; its quotient by the
translation action of $\Gamma\cong\mathbb Z^{2n}$ is a compact real torus. Denote the Riemannian metric tensor on this slice by $h_j$, so that $d_j$ is its length distance. Let
$\mathrm{pr}_1\colon\mathbb R^{2n}\times\mathbb D\to\mathbb R^{2n}$ be the
projection; by~\eqref{eq:deck-action}, $\mathrm{pr}_1(\tilde w)=\tilde x_0$
and $\mathrm{pr}_1((\beta_j,\alpha_j)\cdot\tilde w)=\beta_j\cdot\tilde x_0$.

The $\mathbb D$-component of the lift $\tilde c$ of $\sigma_P(\gamma_j)^\#$
from $\tilde w$ is the lift, from $\tilde y_0$, of the base loop
$f\bigl(\sigma_P(\gamma_j)^\#\bigr)=c_{b_0,b}^{-1}\circ\gamma_j\circ c_{b_0,b}$,
which by Step~2 has image in $N$ and $\rho$-length at most
$L_2:=L_1+2\delta_N$, uniformly in $s,P$. Hence this component is a path of
$\tilde\rho$-length $\le L_2$ starting at $\tilde y_0$ and contained in
$v^{-1}(N)$, so it remains in the closed ball
\[
\Omega:=\overline B_{v^{-1}(N)}(\tilde y_0,L_2)
\]
of radius $L_2$ around $\tilde y_0$ in the connected component of $v^{-1}(N)$ containing $\tilde y_0$, for the induced length metric. This component, with the lifted metric $\tilde\rho$, is a covering of the compact
length space $N$ with the lifted length metric, hence complete and locally
compact, hence a proper metric space by the Hopf--Rinow theorem for length
spaces; therefore $\Omega$ is compact. Note that $\Omega$ is independent of
$s$ and~$P$, and $\tilde c\subset\mathbb R^{2n}\times\Omega$.

The ratio $|\mathrm{pr}_{1*}\xi|_{h_j}/|\xi|_{\tilde h}$, for
nonzero ambient tangent vectors $\xi$ based at points of $\mathbb R^{2n}\times\Omega$, is invariant under~$\Gamma$: every
$\gamma\in\Gamma$ is a $\tilde h$-isometry and satisfies
$\mathrm{pr}_1\circ\gamma=\tau_\gamma\circ\mathrm{pr}_1$, with $\tau_\gamma$
the $d_j$-isometric translation by~$\gamma$. (The analogous equivariance fails
for a general deck element $(\beta,\alpha)$, whose fibre component has linear
part $\varphi_\alpha$ by \eqref{eq:deck-action-general}; this is why we restrict to the $\Gamma$-invariant
region $\mathbb R^{2n}\times\Omega$.) The ratio therefore descends to a
continuous function on the unit sphere bundle over the compact set
$(\mathbb R^{2n}/\Gamma)\times\Omega$, hence it is bounded by some $C_{h,j}\ge 1$; we
set $C_h:=\max_{1\le j\le k}C_{h,j}$, which depends only on $\mathcal A$, $h$,
$\Omega$ and the basis, uniformly in $s,P$. Integrating this differential bound along the actual lifted path $\tilde c\subset\mathbb R^{2n}\times\Omega$ shows that $\mathrm{pr}_1\circ\tilde c$ has
$d_j$-length $\le C_h\,\ell_{\tilde h}(\tilde c)$, and
\begin{equation}\label{eq:fibre-displacement}
  d_j\!\bigl(\beta_j \cdot \tilde{x}_0,\;
    \tilde{x}_0\bigr)
  \;\leq\;
  C_h\,\ell_{\tilde h}(\tilde c)
  \;=\;C_h\,\ell_h\bigl(\sigma_P(\gamma_j)^\#\bigr)
  \;\leq\; C_h\,H(s)\,.
\end{equation}
 
\begin{rem}\label{rmk:displacement}
  The intermediate bound~\eqref{eq:displacement} controls the displacement in
  the \emph{full} universal cover
  $\widetilde{\mathcal{A}}_0 \cong \mathbb{R}^{2n} \times \mathbb D$, involving
  both fibre and base directions (and the metric~$\tilde{h}$ need not be a product metric).  The projection
  step~\eqref{eq:fibre-displacement} extracts a purely fibre-theoretic bound,
  which is what we need for the lattice count in the next step. Only lengths of paths staying in $\mathbb R^{2n}\times\Omega$ are compared; no claim about global Lipschitz continuity for ambient distances is needed.
\end{rem}

\medskip\noindent
\textit{Step~6: Counting.}
For each $j = 1, \ldots, k$, we count the elements $\beta \in \Gamma$
compatible with~\eqref{eq:fibre-displacement}:
\begin{equation}\label{eq:count-set}
  N_j(s)
  \;:=\;
  \#\Bigl\{\beta \in \Gamma :\;
    d_j\!\bigl(\beta \cdot \tilde{x}_0,\;
      \tilde{x}_0\bigr)
    \leq C_h H(s)
  \Bigr\}\,.
\end{equation}
Via the universal covering $u \colon \mathbb{R}^{2n} \to \mathcal A_{b_0}$, the
quotient torus $\mathbb R^{2n}/\Gamma$ is a compact geodesic Riemannian manifold with the metric induced by~$h_j$, and $\Gamma \cong \mathbb{Z}^{2n}$ acts
on~$(\mathbb{R}^{2n}, d_j)$ by deck transformations.  The set
in~\eqref{eq:count-set} is in natural bijection with the set of
$\Gamma$-translates of~$\tilde{x}_0$ lying in the $d_j$-ball of
radius~$C_h H(s)$ around~$\tilde x_0$.

By the fundamental lemma of the geometry of groups
(cf.~\cite[Proposition~A.8 and Lemma~A.9]{Phung}), since
$\Gamma \cong \mathbb{Z}^{2n}$ is a lattice of rank~$2n$ acting cocompactly
on~$(\mathbb{R}^{2n}, d_j)$, there exists $m_j > 0$ (depending
on~$\mathcal{A}$, $h$, $\alpha_j$, but not on~$s$) such that
\begin{equation}\label{eq:lattice-bound}
  N_j(s)
  \;\leq\;
  m_j\,(C_h H(s)+1)^{2n}\,.
\end{equation}
The exponent~$2n$ reflects the rank of $\Gamma \cong \mathbb{Z}^{2n}$: in a
ball of radius~$R$ in the universal cover of the $2n$-dimensional real
torus~$\mathcal A_{b_0}$, the number of $\Gamma$-translates of~$\tilde{x}_0$ grows
as~$R^{2n}$.

Since $G$ is free on $\alpha_1, \ldots, \alpha_k$, a cocycle representative is completely determined by the $k$-tuple $(\beta_1, \ldots, \beta_k)\in\Gamma^k$. Each class arising here has such a bounded representative; distinct tuples can represent the same class, which only decreases the number of classes to be counted.  Bounding each component
independently via~\eqref{eq:lattice-bound} with $R = C_h H(s)$:
\begin{equation}\label{eq:total-splittings}
  \#\{\text{possible cohomology classes } [c_P]\}
  \;\leq\;
  \prod_{j=1}^{k} m_j\,(C_h H(s)+1)^{2n}
  \;=\;
  m_0\,(C_h H(s)+1)^{2nk}\,,
\end{equation}
where $m_0 := \prod_{j=1}^k m_j$.

By Proposition~\ref{prop:parshin-phung}, each class
$[c_P] \in H^1(G, \Gamma)$ arises from at most
$t_\mathcal{A} :=
  \#\bigl(A(K)/\mathrm{Tr}_{K/\mathbb{C}}(A)(\mathbb{C})\bigr)_{\mathrm{tors}}
  < \infty$
rational points modulo the trace (finite by the Lang--N\'eron theorem).
Therefore:
\[
  \#\left( \faktor{I(s,B_0)}{\Tr_{K/\CC}(A)(\mathbb C)}\right)
  \;\leq\;
  t_\mathcal{A} \cdot m_0\,(C_h H(s)+1)^{2nk}
  \;\leq\;
  m\,(s+1)^{2nk}\,,
\]
where
$
  m := t_\mathcal{A} \cdot m_0 \cdot
  \bigl(C_h(c^{-1}L + 2(\delta_0 + \delta_0')) + 1\bigr)^{2nk}
$
depends only on~$B_0$, $\mathcal{A}$, $D$, and~$h$. \qed

\subsection{Improved bound}\label{sec:improved}
We now prove a sharper version of \Cref{thm:main_phung}. The general
strategy is unchanged: we follow the same six-step scheme (reduction to the
hyperbolic locus, loop construction, Parshin cocycle, lifting to
$\mathcal A$, displacement in the universal cover, lattice counting). The
only modification occurs in Steps~2 and~4, where we replace Phung's linear
bound $\ell_S(\gamma_j)=O(s)$ with the sublinear bound
$O(\sqrt{s+1})$ provided by
Proposition~\ref{prop:common-base}. This propagates through the counting
argument and halves the exponent.

\paragraph{Proof of \eqref{eq:main_bound}.}
The proof follows the same six-step scheme as the proof of
\Cref{thm:main_phung}; in particular we may again assume $k\ge1$ and
$s\ge1$, by the same reductions as there.  The main modification is in Step~2: we replace the
appeal to \cite[Theorem~5.1]{Phung}, which provides loops of Kobayashi
length $O(s)$, with Proposition~\ref{prop:common-base}, applied to the set
$S:=f(\sigma_P(B_0)\cap\mathcal D)\cap B_0$ (as in Step~2, we may assume
$S\neq\emptyset$; note $\#S\le s$, and since the bound below is
increasing in the cardinality we may state it with $s$ in place
of~$\#S$), which provides
a point $b\in B_0\setminus S$, a path $\sigma$ from $b_0$ to~$b$ in~$B_0$
with $\ell_\rho(\sigma)\le\delta'$, and loops
$\gamma_1,\ldots,\gamma_k\subset B_0\setminus S$ based at~$b$ with
$[\sigma^{-1}\circ\gamma_j\circ\sigma]=\alpha_j$ in $\pi_1(B_0,b_0)$ and
\[
\ell_S(\gamma_j)\;\le\;C\sqrt{s+1}
\qquad\text{for all }j=1,\ldots,k\,.
\]
Since the loops share a common base point~$b$ and the conjugation uses
the single path~$\sigma$, Steps~3--5 apply verbatim with the conjugating
path $v_b\circ\sigma_O(\sigma)$ in~$\mathcal A_{B_0}$ (in place of
$v_b\circ\sigma_O(c_{b_0,b})$ in \Cref{thm:main_phung}), whose
$h$-length is bounded by $\delta_0+\Vert d\sigma_O\Vert_\infty\delta'$,
independently of~$s$.  The compactness input of Step~5 is again available:
by the last assertion of Proposition~\ref{prop:common-base}, the base loop
$\sigma^{-1}\circ\gamma_j\circ\sigma$ has image in the fixed compact set
$\mathcal K_0\Subset B_0$ and $\rho$-length at most $L_0+2\delta'$,
uniformly in $s,P$; enlarging $N$ so that $\mathcal K_0\subseteq N$ (and updating the fixed bound $\Vert d\sigma_O\Vert_\infty$ on $N$ if necessary), the
compact set $\Omega$ of Step~5 (with $L_2:=L_0+2\delta'$) is independent of
$s,P$.  The displacement bound becomes
\[
d_j\bigl(\beta_j\cdot\tilde x_0,\;\tilde x_0\bigr)
\;\le\;
C_h\,H(s)\,,\qquad
H(s)\;:=\;c^{-1}C\sqrt{s+1}
+2(\delta_0+\Vert d\sigma_O\Vert_\infty\delta')\,,
\]
with $C_h\ge 1$ the constant of Step~5.
The lattice counting in Step~6 then gives
\[
\#\left(\faktor{I(s,B_0)}{\Tr_{K/\CC}(A)(\CC)}\right)
\;\le\;
t_{\mathcal A}m_0\,(C_h H(s)+1)^{2nk}\,.
\]
It remains to extract the exponent. Since $H(s)=O(\sqrt{s+1})$, there is a constant $C'>0$, independent of $s$, such that
\[
(C_hH(s)+1)^{2nk}\le C'(s+1)^{nk}.
\]
Absorbing all constants into $m$, we obtain
\[
\#\left(\faktor{I(s,B_0)}{\Tr_{K/\CC}(A)(\CC)}\right)
\le m(s+1)^{nk}\qquad\text{for all }s\ge0.
\]
The case $s=0$ follows from $I(0,B_0)\subseteq I(1,B_0)$ after enlarging $m$, as in the proof of Phung's bound. There is no logarithmic or $\varepsilon$ loss. This proves \eqref{eq:main_bound}.\qed

\subsection{Counting in the Mordell--Weil lattice and bounding heights}
\label{sec:MW-height}
We retain the fixed reduced base and simple generators used above,
with the improved loop bound of \S\ref{sec:improved}. All choices
in this subsection are made once for this base and family,
independently of $s$ and $P$. The reductions that shrink $B_0$
cause no difficulty: the original set $I(s,B_0)$ is contained
in the set for the reduced base, so every bound proved there
applies to the original points as well.

The additional input retained from the six-step argument is a bound
on a representative of the cohomology class of each point. We first
explain how this controls the point in the Mordell--Weil lattice;
this will give both the rank refinement and the height estimate.

Put $V_\Gamma=\Gamma\otimes\mathbb R$ and first assume $k\ge1$.
Since $G$ is free on $\alpha_1,\ldots,\alpha_k$, arbitrary values
on these generators extend uniquely to a cocycle, so evaluation
identifies
\[
Z^1(G,V_\Gamma)\simeq V_\Gamma^k.
\]
Under this identification the coboundaries form the linear subspace
\[
W=\left\{
\bigl(\varphi_{\alpha_j}(v)-v\bigr)_{j=1}^k:
v\in V_\Gamma
\right\}.
\]
Fix a Euclidean norm $|\cdot|$ on
$V_\Gamma\simeq\mathbb R^{2n}$, and use the product norm
$|(u_j)_j|=(\sum_j|u_j|^2)^{1/2}$ on $V_\Gamma^k$. Give
\[
H^1(G,V_\Gamma)=V_\Gamma^k/W
\]
the quotient norm $\|[u]\|_H=\inf_{w\in W}|u+w|$.
This is a genuine norm because $W$ is a closed linear subspace
of a finite-dimensional vector space. In particular, the norm
of a class is at most the norm of any representative. Moreover,
\[
H^1(G,V_\Gamma)\simeq H^1(G,\Gamma)\otimes\mathbb R:
\]
both sides are the cokernel of the map
$V_\Gamma\to V_\Gamma^k$ displayed above, obtained by tensoring
the integral two-term cochain complex with $\mathbb R$.
Possible torsion in integral cohomology disappears in this
operation; no torsion-freeness assertion is needed.

\paragraph{Proof of \eqref{eq:rank-bound}.}
The induced map
\[
\Psi_{\mathbb R}\colon M\otimes\mathbb R
\longrightarrow H^1(G,\Gamma)\otimes\mathbb R
\]
is injective by Proposition~\ref{prop:parshin-phung}: its integral
kernel is finite, and tensoring with $\mathbb R$ kills that kernel
and preserves exactness. Choose a basis $e_1,\ldots,e_r$ of the
free group $M/M_{\mathrm{tors}}$, and use the Euclidean coordinate
norm
$\|\sum_i a_ie_i\|=(\sum_i a_i^2)^{1/2}$
on $M\otimes\mathbb R$. If $r>0$, the continuous function
$x\mapsto\|\Psi_{\mathbb R}(x)\|_H$ has a positive minimum on
the unit sphere, by injectivity and compactness. Its reciprocal
gives a constant $C_2>0$ such that
\[
\|x\|\le C_2\|\Psi_{\mathbb R}(x)\|_H
\qquad(x\in M\otimes\mathbb R).
\]
For $r=0$ the source is zero and the same inequality is automatic.

For the finitely many slice metrics $h_j$ in Step~5, periodicity
under $\Gamma$ and compactness of the quotient give a constant
$C_7$ such that $|\xi|\le C_7|\xi|_{h_j}$ for every tangent vector
and every $j$. Integrating along any fibre path from $0$ to
$\beta_j$, and then taking the infimum of its $h_j$-length, gives
\[
|\beta_j|
\le C_7d_j(\beta_j\cdot\tilde x_0,\tilde x_0)
\le C_7C_hH(s).
\]
Here $\tilde x_0=0$, as fixed in Step~5, and $C_h$ is obtained
as in that step for the fixed compact set $\Omega$ used in
\S\ref{sec:improved}; it is independent of $s$ and $P$.
For $P\in I(s,B_0)$, the tuple $(\beta_1,\ldots,\beta_k)$
is the representative $c'_P$ of Steps~4--5. The single conjugation
used there ensures that $[c'_P]=\Psi(P)$; thus its real class
is precisely $\Psi_{\mathbb R}(\overline P\otimes1)$,
where $\overline P$ denotes the image in $M$. Consequently
\[
\|\Psi_{\mathbb R}(\overline P\otimes1)\|_H
\le\left(\sum_{j=1}^k|\beta_j|^2\right)^{1/2}
\le\sqrt{k}\,C_7C_hH(s).
\]
Combining the two norm comparisons, with
$C_3=C_2\sqrt{k}\,C_7C_h$, yields
\begin{equation}\label{eq:MW-norm}
\|\overline P\otimes1\|
\le C_3H(s)\le C_4\sqrt{s+1},
\end{equation}
where the fixed additive term in $H(s)$ is absorbed using
$\sqrt{s+1}\ge1$.

The image of $M/M_{\mathrm{tors}}$ is a lattice of rank $r$
in $M\otimes\mathbb R$. In the chosen basis its coordinates
are integers. Equation~\eqref{eq:MW-norm} bounds each coordinate
in absolute value by $C_4\sqrt{s+1}$, so the number of possible
free classes is at most
\[
\bigl(2\lfloor C_4\sqrt{s+1}\rfloor+1\bigr)^r
=O((s+1)^{r/2}).
\]
Each free class has exactly $\#M_{\mathrm{tors}}$ lifts to $M$;
counting only those represented by points of $I(s,B_0)$ can
decrease this number. Multiplying by this fixed finite factor
proves \eqref{eq:rank-bound}, without assuming that $I(s,B_0)$
is invariant under trace translations. For $r=0$ the displayed
count is $1$. For $k=0$, the rank-zero reduction above gives
$M$ finite and the same conclusion directly. The case $s=0$
follows from $I(0,B_0)\subseteq I(1,B_0)$, after enlarging the
constants as before, including $C_4$ in \eqref{eq:MW-norm}.
Finally, injectivity gives
\[
r\le\dim_{\mathbb R}H^1(G,V_\Gamma)
\le\dim_{\mathbb R}V_\Gamma^k=2nk,
\]
which explains the relation between the exponents $r/2$
and $nk$.\qed

\paragraph{Proof of Corollary~\ref{cor:height}.}
The canonical height of a symmetric ample line bundle over a
complex function field is invariant under translation by the
constant trace and vanishes on torsion; it descends to a quadratic
form on $M\otimes\mathbb R$; see
\cite[\S9, especially Lemma~9.13]{ConradTrace}.
In the basis chosen above, this form is given by a fixed
symmetric real matrix $Q$: if $\overline P$ has free coordinates
$(a_1,\ldots,a_r)$, then
$\widehat h_L(P)=a^{\mathsf T}Qa$.
Taking an upper bound for this quadratic form on the Euclidean
unit sphere gives $C_5>0$ with
\[
\widehat h_L(P)\le C_5\|\overline P\otimes1\|^2.
\]
Equation~\eqref{eq:MW-norm} therefore proves
\[
\widehat h_L(P)\le C_5C_4^2(s+1)
\qquad(P\in I(s,B_0)).
\]
Thus we may take $C=C_5C_4^2$. This uses the bound on the coordinates
of each point, rather than the cardinality estimate: a bound
on the number of lattice points by itself would give no bound
on their distance from the origin. If $r=0$ (in particular,
if $k=0$), the height is identically zero and the conclusion
holds directly with any $C>0$.

Suppose now that $\mathcal D$ is Cartier. The function
\[
h_{\mathcal D}(P)
=\deg_B\sigma_P^*\mathcal O_{\mathcal A}(\mathcal D)
\]
is a Weil height attached to $\mathcal O_A(D)$ for this fixed
model. Choose an integer $a>0$ such that
$E:=L^{\otimes a}\otimes\mathcal O_A(-D)$ is ample.
An ample Weil height $h_E$ over $K$ is bounded below.
The additivity of Weil heights and the bounded difference
between a Weil height $h_L$ for $L$ and $\widehat h_L$ give
\[
h_E(P)=a h_L(P)-h_{\mathcal D}(P)+O(1),
\qquad
h_L(P)=\widehat h_L(P)+O(1),
\]
where the bounded terms are uniform in $P$. Rearranging,
we obtain
\[
h_{\mathcal D}(P)\le a\widehat h_L(P)+C_6
\]
with $C_6$ depending on the fixed height data and model;
compare the argument of \cite[Corollary~1.3]{Phung}.
For $P\in I(s,B_0)$ the section is not contained in $\mathcal D$,
and this line-bundle degree equals the degree of the pulled-back
effective divisor. Combining the inequalities and using
$s+1\ge1$, we conclude that
\[
\deg_B\sigma_P^*\mathcal D
\le aC(s+1)+C_6
\le\bigl(aC+\max\{C_6,0\}\bigr)(s+1).
\]
The same comparison covers the rank-zero case.\qed

\begin{rem}
The quotient norms and the inverse bound for $\Psi_{\mathbb R}$
are chosen for the fixed local system and base. Their constants
need not be uniform when the family, model or base changes.
Thus the linear height bound above is a fixed-data result
under (P)(iii), even though it applies to intersections allowed
to be unrestricted over the removed disks.
\end{rem}

\Addresses
\end{document}